\documentclass[11pt]{tran-l}

\usepackage{amssymb}
\usepackage{graphicx}
\usepackage{color}

\newtheorem{theorem}{Theorem}[section]
\newtheorem{lemma}[theorem]{Lemma}

\theoremstyle{definition}

\theoremstyle{remark}

\numberwithin{equation}{section}

\renewcommand{\Gamma}{\varGamma}
\renewcommand{\epsilon}{\varepsilon}

\renewcommand{\hat}{\widehat}

\renewcommand{\leq}{\leqslant}
\renewcommand{\geq}{\geqslant}

\newcommand{\K}{\mathcal{K}}
\newcommand{\E}{\mathbb{E}^3}
\def\apeir{\mathop{\rm apeir}}

\begin{document}

\title[Polygonal Complexes]{Regular Polygonal Complexes in Space, II}


\author{Daniel Pellicer}
\address{Instituto de Matematicas, Unidad Morelia, CP 58089, Morelia, Michoacan, Mexico}
\curraddr{}
\email{pellicer@matmor.unam.mx}
\thanks{}

\author{Egon Schulte}
\address{Northeastern University,
Boston, MA 02115, USA}
\curraddr{}
\email{schulte@neu.edu}
\thanks{Supported by NSF grant DMS--0856675}


\date{}


\begin{abstract}
Regular polygonal complexes in euclidean $3$-space $\E$ are discrete polyhedra-like structures with finite or infinite polygons as faces and with finite graphs as vertex-figures, such that their symmetry groups are transitive on the flags. The present paper and its predecessor describe a complete classification of regular polygonal complexes in $\E$. In Part I we established basic structural results for the symmetry groups, discussed operations on their generators, characterized the complexes with face mirrors as the $2$-skeletons of the regular $4$-apeirotopes in $\E$, and fully enumerated the simply flag-transitive complexes with mirror vector~$(1,2)$. In this paper, we complete the enumeration of all regular polygonal complexes in $\E$ and in particular describe the simply flag-transitive complexes for the remaining mirror vectors. It is found that, up to similarity, there are precisely 25 regular polygonal complexes which are not regular polyhedra, namely 21 simply flag-transitive complexes and $4$ complexes which are $2$-skeletons of regular $4$-apeirotopes in $\E$.
\end{abstract}

\maketitle


\section{Introduction}

The present paper and its predecessor \cite{pelsch} describe a complete classification of regular polygonal complexes in the euclidean $3$-space $\E$. Polygonal complexes are discrete polyhedra-like structures composed of convex or non-convex, planar or skew, finite or infinite (helical or zigzag) polygonal faces, always with finite graphs as vertex-figures, such that each edge lies in at least two, but generally $r\geq 2$ faces, with $r$ not depending on the particular edge. The various kinds of $3$-dimensional polyhedra that have been studied in the literature are prominent examples of polygonal complexes, obtained when $r=2$ (see Coxeter~\cite{crsp,coxeter}, Gr\"unbaum~\cite{gr1} and McMullen \& Schulte~\cite{arp}). A polygonal complex is {\em regular\/} if its full euclidean symmetry group is transitive on the flags.

Our two papers are part of an ongoing program that combines a {\em skeletal\/} approach to polyhedra in space pioneered in \cite{gr1} (see also Dress~\cite{d1,d2} and McMullen \& Schulte~\cite{ordinary}), with an effort to study symmetry of discrete polyhedra-like space structures through transitivity properties of their symmetry group. The full enumeration of the chiral polyhedra in $\E$ in \cite{chiral1, chiral2} (see also Pellicer \& Weiss~\cite{pelwei}), as well as a number of corresponding enumeration results for figures in higher-dimensional euclidean spaces by McMullen~\cite{ pm,pm1,pm2} (see also Arocha, Bracho \& Montejano~\cite{ar} and \cite{bra}), are examples of recent successes of this program; for a survey, see~\cite{ms3}.

The paper is organized as follows. In Section~\ref{terba} we review basic properties of regular polygonal complexes and their symmetry groups, and elaborate on two important operations that produce new regular complexes from old. Then in Sections~\ref{mirr11}, \ref{mir0k} and \ref{mir2k}, respectively, we enumerate the simply flag-transitive regular polygonal complexes with mirror vectors $(1,1)$, $(0,k)$ and $(2,k)$, with $k=1$ or $2$. In Section~\ref{mirrcyc} we eliminate the possibility that a simply flag-transitive regular polygonal complex with mirror vector $(0,k)$ has pointwise edge stabilizers that are cyclic of order $r\geq 3$. Together with the results of \cite{pelsch}, our findings complete the enumeration of all regular polygonal complexes in $\E$. Overall we establish that, up to similarity, there are precisely 25 regular polygonal complex in $\E$ which are not regular polyhedra, namely 21 simply flag-transitive complexes and $4$ complexes which are $2$-skeletons of regular $4$-apeirotopes in $\E$.

\section{Terminology and basic facts}
\label{terba}

A {\em finite polygon\/} $(v_1, v_2, \dots, v_n)$ in euclidean $3$-space $\E$ is a figure formed by distinct points $v_1, \dots, v_n$, together with the line segments $(v_i, v_{i+1})$, for $i = 1, \dots, n-1$, and $(v_n, v_1)$. Similarly, an {\em infinite polygon\/} consists of a sequence of distinct points $(\dots, v_{-2},v_{-1}, v_0, v_1, v_2,\dots)$ and of line segments $(v_i, v_{i+1})$ for each $i$, such that each compact subset of $\E$ meets only finitely many line segments. In either case the points and line segments are the {\em vertices\/} and {\em edges\/} of the polygon, respectively.

A {\em polygonal complex}, or simply {\em complex}, $\K$ in $\E$ consists of a set $\mathcal{V}$ of points, called {\em vertices}, a set $\mathcal{E}$ of line segments, called {\em edges}, and a set $\mathcal{F}$ of polygons, called {\em faces}, such that the following properties are satisfied. The graph defined by $\mathcal{V}$ and $\mathcal{E}$, called the {\em edge graph\/} of $\K$, is connected. Moreover, the vertex-figure of $\K$ at each vertex of $\K$ is connected. Recall that the {\em vertex-figure\/} of $\K$ at a vertex $v$ is the graph, possibly with multiple edges, whose vertices are the neighbors of $v$ in the edge graph of $\K$ and whose edges are the line segments $(u,w)$, where $(u, v)$ and $(v, w)$ are edges of a common face of $\K$.  It is also required that each edge of $\K$ is contained in exactly $r$ faces of $\K$, for a fixed number $r \geq 2$. Finally, $\K$ is {\em discrete\/}, in the sense that each compact subset of $\E$ meets only finitely many faces of $\K$.

A complex with $r=2$ is also called a {\em polyhedron\/}. Finite or infinite polyhedra in $\E$ with high symmetry properties have been studied extensively (for example, see \cite[Ch.~7E]{arp} and \cite{gr1,ordinary,pelwei,chiral1,chiral2}).

A polygonal complex $\K$ is said to be (geometrically) {\em regular} if its symmetry group $G:=G(\K)$ is transitive on the flags (triples consisting of a vertex, an edge, and a face, all mutually incident). We simply refer to the (full) symmetry group $G$ as the {\em group of $\K$}. If $\K$ is regular, its faces are necessarily regular polygons, either finite, planar (convex or star-) polygons or non-planar ({\em skew\/}) polygons, or infinite, planar zigzags or helical polygons. Moreover, its vertex-figures are graphs with single or double edges; the latter occurs precisely when any two adjacent edges of a face of $\K$ are adjacent edges of just one other face of $\K$. We know from \cite{pelsch} that, apart from polyhedra, there are no regular complexes that are finite or have an affinely reducible group.

Let $\K$ be a regular complex, and let $G$ be its group. Let $\Phi := \{F_0, F_1, F_2\}$ be a fixed, or {\em base}, flag of $\K$, where $F_0$ is a vertex, $F_1$ an edge, and $F_2$ a face of $\K$.  If $\Psi$ is a subset of $\Phi$, we let
$G_{\Psi}$ denote its stabilizer in $G$. For $i=0,1,2$ we also set $G_i := G_{\{F_j, F_k\}}$, where $i,j,k$ are distinct.
We showed in \cite{pelsch} that $|G_\Phi|\leq 2$. Thus the group of a regular complex either acts simply flag-transitively or has flag-stabilizers of order $2$. We call $\K$ {\em simply flag-transitive\/} if its (full symmetry) group $G$ acts simply flag-transitively on $\K$.

In \cite{pelsch} we characterized the regular complexes with non-trivial flag stabilizers as the $2$-skeletons of regular $4$-apeirotopes in $\E$. These complexes have planar faces and have {\em face mirrors\/}, the latter meaning that the affine hull of a face is the mirror (fixed point set) of a plane reflection in $G$. There are eight regular $4$-apeirotopes in $\E$; however, since a pair of Petrie-duals among these apeirotopes share the same $2$-skeleton, these only yield four regular complexes $\K$. We can list the eight $4$-apeirotopes in a more descriptive way in four pairs of Petrie duals using the notation of~\cite{arp}.
\smallskip
\begin{equation}
\label{4apeirotopes}
\begin{array}{cc}
\{4, 3, 4\} &\{\{4, 6 \,|\,4\}, \{6, 4\}_3\}\\[.04in]
\apeir \{3, 3\} \!=\! \{\{\infty, 3\}_6 \# \{ \, \}, \{3, 3\}\}
&\{\{\infty, 4\}_4 \# \{\infty\}, \{4, 3\}_3\} \!=\! \apeir\{4, 3\}_3 \\[.04in]
\apeir \{3, 4\} \!=\! \{\{\infty, 3\}_6 \# \{ \, \}, \{3, 4\}\}
&\{\{\infty, 6\}_3 \# \{\infty\}, \{6, 4\}_3\} \!=\! \apeir\{6, 4\}_3 \\[.04in]
\apeir \{4, 3\} \!=\! \{\{\infty, 4\}_4 \# \{\, \}, \{4, 3\}\}
&\{\{\infty, 6\}_3 \# \{\infty\}, \{6, 3\}_4\} \!=\! \apeir \{6, 3\}_4\\[.04in]
\end{array}
\end{equation}
The apeirotopes in the top row are the cubical tessellation $\{4,3,4\}$ and its Petrie dual; these have square faces, and their facets are cubes or Petrie-Coxeter polyhedra $\{4, 6 \,|\,4\}$, respectively. All other apeirotopes have zigzag faces, and their facets are blends of the Petrie-duals $\{\infty,3\}_6$, $\{\infty,6\}_3$ or $\{\infty,4\}_4$ of the plane tessellations $\{6,3\}$, $\{3,6\}$ or $\{4,4\}$, respectively, with the line segment $\{ \,\}$ or linear apeirogon $\{\infty\}$ (see \cite[Ch.~7F]{arp}). These six apeirotopes can be obtained as particular instances from the {\em free abelian apeirotope\/} or ``apeir" construction of \cite{pm,pm1}, which we briefly review here for rank $4$.

Let $\mathcal Q$ be a finite regular polyhedron in $\E$ with symmetry group $G(\mathcal{Q}) = \langle T_1,T_2,T_3\rangle$ (say), where the labeling of the distinguished generators begins at $1$ deliberately. Let $o$ be the centroid of the vertex-set of $\mathcal Q$, let $w$ be the initial vertex of $\mathcal Q$, and let $T_0$ denote the reflection in the point $\frac{1}{2}w$. Then there is a regular $4$-apeirotope in $\E$, denoted $\apeir {\mathcal Q}$, with $T_0,T_1,T_2,T_3$ as the generating reflections of its symmetry group, $o$ as initial vertex, and $\mathcal Q$ as vertex-figure. In particular, $\apeir Q$ is discrete if $\mathcal Q$ is rational (the vertices of $\mathcal Q$ have rational coordinates with respect to some coordinate system). The latter limits the choices of $\mathcal Q$ to $\{3,3\}$, $\{3,4\}$ or $\{4,3\}$, or their Petrie duals $\{4,3\}_3$, $\{6,4\}_3$ or $\{6,3\}_4$, respectively, giving the six remaining regular $4$-apeirotopes in $\E$.

The enumeration of the simply flag-transitive regular complexes is a lot more involved. From now on, unless specified otherwise, we will work under the standard assumption that the complexes $\K$ under consideration are infinite, regular, and simply flag-transitive, and have an affinely irreducible group $G(\K)$.

Thus let $\K$ be an (infinite) simply flag-transitive regular complex, and let $G=G(\K)$ be its (affinely irreducible) group.
We know from \cite{pelsch} that $G_0 = \langle R_0 \rangle$ and $G_1 = \langle R_1 \rangle$, for some point, line or plane reflection $R_0$ and some line or plane reflection $R_1$; moreover, $G_2$ is a cyclic or dihedral group of order $r$ (so $r$ is even if $G_2$ is dihedral). The {\em mirror vector} of $\K$ is the vector $(dim(R_0), dim(R_1))$, where $dim(R_i)$ is the dimension of the mirror of the reflection $R_i$ for $i=0,1$; if $\K$ is a polyhedron, then $G_2$ is generated by a (line or plane) reflection $R_2$ and we refer to $(dim(R_0), dim(R_1),dim(R_2))$ as the {\em complete mirror vector} of $\K$. The face stabilizer subgroup $G_{F_2}$ in $G$ of the base face $F_2$ is given by $G_{F_2}=\langle R_0, R_1 \rangle$ and is isomorphic to a (finite or infinite) dihedral group acting simply transitively on the flags of $\K$ containing $F_2$. Similarly, the vertex-stabilizer subgroup $G_{F_0}$ in $G$ of the base vertex $F_0$ is given by $G_{F_0}=\langle R_1, G_2 \rangle$ and acts simply flag-transitively on (the graph that is) the vertex-figure of $\K$ at $F_0$. (A flag in the vertex-figure of $\K$ at $F_0$ amounts to a pair consisting of an edge and  incident face of $\K$ each containing $F_0$.)  We call $G_{F_0}$ the {\em vertex-figure group\/} of $\K$ at $F_0$. Note that, by our discreteness assumption on complexes, $G_{F_0}$ must be a finite group.

In our previous paper~\cite{pelsch} we already dealt with the complexes with mirror vector $(1,2)$. In this paper, we complete the enumeration of the simply flag-transitive regular complexes and describe the complexes for the remaining mirror vectors. Our approach employs operations on the generators of $G$ which replace one of the generators $R_0$ or $R_1$ while retaining the other as well as the subgroup $G_2$.  This allows us to construct new complexes from old and helps reduce the number of cases to be considered. In particular, we require the following two operations $\lambda_0$ and $\lambda_1$ that involve (not necessarily involutory) elements $R$ of $G_{2}$ with the property that $R_0 R$ or $R_1 R$, respectively, is an involution:
\begin{equation}
\label{opone}
\lambda_0 = \lambda_0(R)\!:\;\,  (R_0, R_1, G_2)\; \mapsto\; (R_0 R, R_1, G_2),
\end{equation}\\[-.45in]
\begin{equation}
\label{optwo}
\lambda_1 = \lambda_1(R) : \:\, (R_0, R_1, G_2)\; \mapsto\; (R_0, R_1R, G_2).
\end{equation}
The corresponding complexes $\K^{\lambda_0}$ and $\K^{\lambda_1}$ are obtained from Wythoff's construction applied with the generators and generating subgroups on the right-hand side of (\ref{opone}) or (\ref{optwo}), respectively.

The two operations in (\ref{opone}) and (\ref{optwo}) can also be applied to regular complexes with face-mirrors by choosing as $R_0$ or $R_1$, respectively, particular   elements of $G_0$ or $G_1$ that do not stabilize the base flag. (Note that we cannot change the entire subgroup $G_0$ or $G_1$, respectively, to its coset $G_{0}R$ or $G_{1}R$, since this is not even a group; instead we must work with particular elements of $G_0$ or $G_1$.)  

In this wider setting of arbitrary regular complexes, the operations $\lambda_0 = \lambda_0(R)$ and $\lambda_1 = \lambda_1(R)$ are invertible at least at the level of groups (but not at the level of complexes in general); in fact, at the  group level, their inverses are given by $\lambda_0(R^{-1})$ and $\lambda_1(R^{-1})$, respectively. While the invertibility of the operations at the level of the corresponding complexes will be immediately clear when the new complex $\K^{\lambda_0}$ or $\K^{\lambda_1}$ is simply flag-transitive, more care is required when the new complex has face-mirrors.  

After we apply an operation $\lambda_0 = \lambda_0(R)$ or $\lambda_1 = \lambda_1(R)$ at the level of (arbitrary) regular complexes, we may arrive at a new complex $\K^{\lambda_0}$ or $\K^{\lambda_1}$ with face mirrors. In this case the (involutory) element $R_{0}R$ of $G_{0}(\K^{\lambda_0})$ or $R_{1}R$ of $G_{1}(\K^{\lambda_1})$ is available as a particular choice of generator to base the inverse operation $\lambda_0(R^{-1})$ or $\lambda_1(R^{-1})$ on (this would have been the only possible choice had $\K^{\lambda_0}$ or $\K^{\lambda_1}$ been simply flag-transitive). While this choice may not directly recover the original complex $\K$ from $\K^{\lambda_0}$ or $\K^{\lambda_1}$, it does produce a regular complex $\mathcal{L}$ containing $\K$ as a (possibly proper) subcomplex. Throughout, we are adopting  the {\em convention\/} to base the inverse operation on the particular element $R_{0}R$ or $R_{1}R$ of its respective subgroup. Note that, when $\K^{\lambda_0}$ or $\K^{\lambda_1}$ has face mirrors, there would have been just one other admissible choice for the particular element besides $R_{0}R$ or $R_{1}R$ (the respective subgroup and the flag stabilizer are isomorphic to $C_{2}\times C_2$ and $C_2$, respectively). 

In our applications, $R$ will always be an involution in $G_2$ and the corresponding operation $\lambda_0$ or $\lambda_1$ will be involutory as well. In particular, we will encounter statements of the form $\K=(\K^{\lambda_0})^{\lambda_0}$ or $\K=(\K^{\lambda_1})^{\lambda_1}$, where throughout an appropriate interpretation (following our convention) is understood if a complex happens to have face mirrors. As we will see, in practice it is only $\lambda_0$ that requires special consideration for complexes with face mirrors (and in only one case).

The following lemmas summarize basic properties of $\lambda_0$ and $\lambda_1$ (see \cite[Lemmas 5.1--5.5]{pelsch}. The first two are saying that the new generators on the right side of (\ref{opone}) and (\ref{optwo}) indeed determine a new regular complex in each case.

\begin{lemma}
\label{lambda1}
Let $\K$ be a simply flag-transitive regular complex with group $G = \langle R_0, R_1, G_2 \rangle$, and let $R$ be an element in $G_2$ such that $R_0 R$ is an involution. Then there exists a regular complex, denoted $\K^{\lambda_0}$, with the same vertex-set and edge-set as $\K$ and with its symmetry group containing $G$ as a (possibly proper) flag-transitive subgroup, such that
\begin{equation}\label{klambda}
\langle R_0 R \rangle \subseteq G_0(\K^{\lambda_0}), \;\;\,
  G_1(\K) = \langle R_1 \rangle \subseteq G_1(\K^{\lambda_0}),\;\;\,
  G_2(\K) \subseteq G_2(\K^{\lambda_0}).
\end{equation}
The complex $\K^{\lambda_0}$ is simply flag-transitive  if and only if the inclusions in (\ref{klambda}) are equalities (or equivalently, at least one of the inclusions in (\ref{klambda}) is an equality).
\end{lemma}

Lemma~\ref{lambda1} is a slightly revised version of Lemma~5.1 in~\cite{pelsch}, which was incorrect as stated. As pointed out on \cite[p.\! 6692]{pelsch}, there are examples where the new complex $\K^{\lambda_0}$ is not simply flag-transitive but rather has face-mirrors and possibly a strictly larger symmetry group; the latter depends on whether or not the reflections in the face-mirrors of $\K^{\lambda_0}$ are also symmetries of $\K$ (see Section~\ref{k02}). However, by mistake, this possibility was not carried forward to the wording of Lemma~5.1 in~\cite{pelsch}. Our new version corrects this error. Similarly, our Lemmas~\ref{appone} and \ref{apptwo} below are slightly revised versions of corresponding statements in \cite{pelsch}, with the only adjustments directly resulting from those in Lemma~\ref{lambda1}.

By contrast, the simple flag-transitivity is preserved in our next lemma, which describes the effect of the operation $\lambda_1$. In fact, we proved in \cite{pelsch} that~$\lambda_1$, applied to a regular complex with face mirrors, always yields another regular complex with face mirrors.

\begin{lemma}
\label{lambda2}
Let $\K$ be a simply flag-transitive regular complex with group $G = \langle R_0, R_1, G_2 \rangle$, and let $R$ be an element in $G_2$ such that $R_1 R$ is an involution. Then there exists a regular complex, denoted $\K^{\lambda_1}$ and again simply flag-transitive, with the same vertex-set and edge-set as $\K$ and with the same group $G$, such
that
\begin{equation}\label{klambda1}
G_0(\K) = \langle R_0 \rangle = G_0(\K^{\lambda_1}), \;\;\,
\langle R_1 R \rangle = G_1(\K^{\lambda_1}),\;\;\,
G_2(\K) = G_2(\K^{\lambda_1}).
\end{equation}
\end{lemma}
\smallskip

The next three lemmas state that the operations $\lambda_0$ and $\lambda_1$ change the mirror vectors in a uniform way, that is, independent of $\K$ (but possibly dependent on whether $G_2$ is dihedral or cyclic). For the first  two lemmas recall our convention about the double iteration of $\lambda_0$ if the new complex $\K^{\lambda_0}$ happens to have face mirrors.

\begin{lemma}
\label{appone}
Let $\K$ be an infinite simply flag-transitive regular complex with an affinely irreducible group $G$ and mirror vector $(2, k)$ for some $k = 1, 2$. Then $G_2$ contains a half-turn $R$. In particular, the corresponding complex $\K^{\lambda_0}$, with $\lambda_0 = \lambda_0(R)$, is a regular complex which either has face mirrors or is simply flag-transitive with mirror vector $(0, k)$; in either case, $\K = (\K^{\lambda_0})^{\lambda_0}$.
\end{lemma}

\begin{lemma}
\label{apptwo}
Let $\K$ be an infinite simply flag-transitive regular complex with an affinely irreducible group $G$, a dihedral subgroup $G_2$, and mirror vector $(0, k)$ for some $k = 1, 2$. Then, for any plane reflection $R \in G_2$, the corresponding complex $\K^{\lambda_0}$, with $\lambda_0 = \lambda_0(R)$, is a regular complex which either has face mirrors or is simply flag-transitive with mirror vector $(1, k)$; in either case, $\K = (\K^{\lambda_0})^{\lambda_0}$.
\end{lemma}

\begin{lemma}
\label{appthree}
Let $\K$ be an infinite simply flag-transitive regular complex with an affinely irreducible group $G$ and mirror vector $(k, 1)$ for some $k = 0, 1, 2$. Assume also that $G_2$ contains a plane reflection $R$ whose mirror contains the axis of the half-turn $R_1$. Then the corresponding complex $\K^{\lambda_1}$, with $\lambda_1 = \lambda_1(R)$, is a simply flag-transitive regular complex with mirror vector $(k, 2)$. In particular, $\K = (\K^{\lambda_1})^{\lambda_1}$.
\end{lemma}

As mentioned earlier, the symmetry group of a simply flag-transitive regular complex $\K$ may be only a proper subgroup of the symmetry group of the new complex $\K^{\lambda_0}$. Clearly, this can only occur if $\K^{\lambda_0}$ itself is not simply flag-transitive. Now under the assumptions of Lemmas \ref{appone} and \ref{apptwo}, the given simply flag-transitive complex $\K$ must necessarily have planar faces. Hence, if $\K^{\lambda_0}$ acquires face mirrors, then these face mirrors must necessarily be the affine hulls of the faces of $\K$; bear in mind here that the geometry of the base face of $\K^{\lambda_0}$ is entirely determined by the subgroup $\langle R_{0}R,R_1\rangle$, and that therefore this base face lies in the same plane as the base face of $\K$. The equality $\K = (\K^{\lambda_0})^{\lambda_0}$ at the end of Lemmas \ref{appone} and \ref{apptwo} follows from this argument. 

When the new $\K^{\lambda_0}$ is not simply flag-transitive, the question arises whether or not the reflective symmetries in the face mirrors of $\K^{\lambda_0}$ are also symmetries of~$\K$. Here $\K^{\lambda_0}$ has a strictly larger symmetry group than $\K$ precisely when the face mirrors of $\K^{\lambda_0}$ are {\em not\/} face mirrors of $\K$ (that is, when the reflective symmetries in the face mirrors of $\K^{\lambda_0}$ are not symmetries of~$\K$). For example, the $2$-skeleton of the regular $4$-apeirotope $\apeir \{3,4\}$, viewed as the complex $\K_{4}(1,2)^{\lambda_0}$ as described in Section~\ref{k02}, has a strictly larger symmetry group than the original (simply flag-transitive) complex $\K_{4}(1,2)$. 
\smallskip

The symmetry groups of regular complexes are crystallographic groups. Recall that the {\em special group\/} $G_*$ of a crystallographic group $G$ is the image of $G$ under the epimorphism $\mathcal{I}(3)\mapsto \mathcal{O}(3)$ whose kernel is $\mathcal{T}(\E)$; here $\mathcal{I}(3)$, $\mathcal{O}(3)$, and $\mathcal{T}(\E)$, respectively, are the euclidean isometry group, orthogonal group, and translation group of $\E$.  Then $G_*$ is finite and
$G_*= G/(G \cap \mathcal{T}(\E)) = G/T(G)$, where $T(G)$ is the full translation subgroup of $G$ (which may be identified with a lattice in $\E$). More explicitly, if $R: x \mapsto xR' + t$ is any element of $G$, with $R' \in \mathcal{O}(3)$ and $t \in \E$, then $R'$ lies in $G_*$; conversely, all elements of $G_*$ are obtained in this way from elements in $G$.
\smallskip

Let $a$ be a positive real number, let $k=1$, $2$ or $3$, and let ${\bf a} : = (a^k,0^{3-k})$, the vector with $k$ components $a$ and $3-k$ components $0$. Following \cite[p.166]{arp}, we write $\Lambda_{\bf a}$ for the sublattice of $a\mathbb{Z}^3$ generated by $\bf a$ and its images under permutation and changes of sign of coordinates. Then $\Lambda_{(1,0,0)}=\mathbb{Z}^{3}$ is the standard {\em cubic lattice\/}; $\Lambda_{(1,1,0)}$ is the {\em face-centered cubic lattice\/} consisting of all integral vectors with even coordinate sum; and $\Lambda_{(1,1,1)}$ is the {\em body-centered cubic\/} lattice.

The geometry of a number of regular complexes described later can best be described in terms of the semiregular tessellation $\mathcal{S}$ of $\E$ by regular tetrahedra and octahedra constructed as follows (see \cite{coxsr,gruni,jo}). Let $Q$ denote the regular octahedron with vertices $(\pm a,0,0)$, $(0,\pm a,0)$, $(0,0,\pm a)$, and let $\mathcal{Q}$ denote the family of octahedra $u+Q$ centered at the points $u$ in $a\mathbb{Z}^3$ not in $\Lambda_{(a,a,0)}$. Then the complement in $\E$ of the octahedra in $\mathcal{Q}$ gives rise to a family $\mathcal{R}$ of regular tetrahedra each inscribed in a cube of the cubical tessellation with vertex-set $a\mathbb{Z}^3$; each such cube $C$ contributes just one tetrahedron $T_C$ to $\mathcal{R}$, and the tetrahedra in adjacent cubes share an edge with $T_C$. The family $\mathcal{Q}\cup\mathcal{R}$ of octahedra and tetrahedra consists of the tiles in a tessellation $\mathcal{S}$ of $\E$; this tessellation is {\em semiregular\/}, meaning that its tiles are Platonic solids and the symmetry group of $\mathcal{S}$ acts transitively on the vertices of $\mathcal{S}$. The faces of $\mathcal{S}$ are regular triangles, each a face of one octahedron and one tetrahedron.
\smallskip

Concluding this section we record the following simple lemma without proof.

\begin{lemma}
\label{cubelemma}
Let $C$ be a cube, and let $R,R',R''$ be symmetries of $C$ such that $R$ and $R'$ are plane reflections and $R''$ is a half-turn. Suppose one of the following conditions applies: the mirrors of $R$ and $R'$ are determined by the two diagonals of a face $F$ of $C$, and $R''$ is a half-turn whose axis passes through the midpoint of an edge of $F$; or $R$ and $R'$ are the two reflections leaving an edge $E$ of $C$ invariant, and $R''$ is a half-turn whose axis passes through the midpoint of an edge adjacent to $E$. Then $R,R',R''$ generate the full octahedral group $G(C)=[3,4]$.
\end{lemma}

\section{Complexes with mirror vector $(1, 1)$}
\label{mirr11}

In this section we enumerate the infinite simply flag-transitive regular complexes with mirror vector $(1, 1)$, exploiting Lemma~\ref{appthree} and drawing on the enumeration of the regular complexes with mirror vector $(1,2)$ in \cite{pelsch}. As we shall see, all have helical faces. We will work again under the assumption that the symmetry group is affinely irreducible. It is known that there are exactly six regular complexes of this kind which are polyhedra, all with helices as faces:\ in the notation of \cite[Section 7E]{arp} these are $\{\infty,3\}^{(a)}$, $\{\infty,4\}_{\cdot,\star 3}$, $\{\infty,3\}^{(b)}$ with complete mirror vector $(1,1,1)$ and $\{\infty, 6\}_{4,4}$, $\{\infty, 4\}_{6,4}$, $\{\infty, 6\}_{6,3}$ with complete mirror vector $(1,1,2)$. We generally take the enumeration of the regular polyhedra for granted and concentrate on the complexes which are not polyhedra.

For the sake of simplicity, whenever we claim uniqueness for a choice of certain elements within a group (or its special group) or of mirrors of such elements, we will usually omit any qualifying statements such as ``up to conjugacy'' or ``up to congruence''. Throughout, these qualifications are understood.

Throughout, let $\K$ be an infinite simply flag-transitive regular complex with mirror vector $(1,1)$ and an affinely irreducible symmetry group $G = \langle R_0, R_1, G_2 \rangle$, where the subgroup $G_2$ has order $r\geq 3$.

The isometries $R_0$ and $R_1$ are half-turns whose axes, $L_0$ and $L_1$ respectively, either intersect at the center of the base face $F_2$ if $\K$ has finite faces, or do not intersect at all if $\K$ has infinite faces. Their product $R_0 R_1$ is a twist, with a trivial or non-trivial translation component, whose invariant line $L_3$ is perpendicular
to the axes of $R_0$ and $R_1$. The subgroup $G_2$ fixes the line $L_2$ through $F_1$ pointwise, and the generator $S$ of its rotation subgroup is a rotation about $L_2$. Note here that $L_0$ is perpendicular to $L_2$. On the other hand, $L_2$ cannot be perpendicular to $L_1$ or parallel to $L_3$, since otherwise $F_2$ would necessarily have to be a linear apeirogon. For the same reason, $L_0$ and $L_1$ are not parallel. Moreover, since $R_1$ does not fix $F_1$, its axis $L_1$ cannot coincide with $L_2$.

It is convenient to assume that $o$ is the base vertex of $\K$. Then $o$ is fixed by $R_1$ and each element of $G_2$. It follows that the vertex-figure group $\langle R_1, G_2 \rangle$ of $\K$ at $o$ is a (possibly proper) subgroup of the special group $G_*$. Recall that, if $R$ is any element of $G$ or $L$ is any line, we let $R'$ denote the image of $R$ in $G_*$ and $L'$ the translate of $L$ through~$o$.

The special group $G_*$ is a finite irreducible crystallographic subgroup of $\mathcal{O}(3)$ that contains the three distinct rotations $R_0'$, $R_1$ and $S$, whose axes $L_0'$, $L_1$ and $L_2$ are positioned in such a way that $L_2$ is perpendicular to $L_0'$ but not to $L_1$. This immediately rules out the groups $[3, 3]^+$, $[3, 3]^*$ and $[3, 3]$ as special groups of $G$. In fact, the rotation subgroup of these groups is $[3, 3]^+$ in each case; however, then the axis of a non-involutory rotation like $S$ could not be perpendicular to the axis of an involutory rotation like $R_{0}'$. Hence $G_*$ is either $[3, 4]^+$ if $G_2$ is cyclic, or $[3, 4]$ if $G_2$ is dihedral.

We now proceed to determine the regular complexes with mirror vector $(1, 1)$. As the reference figure for the action of $G_*$ we take the cube $C$ with vertices $(\pm 1, \pm 1, \pm 1)$.\\[-.3in]

\subsection{The four complexes derived through $\lambda_1$}
\label{lam}

First we employ the operation $\lambda_1$ described in (\ref{optwo}) and Lemma~\ref{appthree} to obtain those complexes for which the axis of $R_1$ is contained in a mirror of a plane reflection in $G_2$ (this corresponds to case (A) in \cite[\S 6.1]{pelsch}). According to Lemma~\ref{appthree} applied with $k=1$, we need to apply $\lambda_1$ to those regular, simply flag-transitive complexes with mirror vector $(1, 2)$ for which the mirror of the corresponding plane reflection $R_1$ is perpendicular to the mirror of a plane reflection in the corresponding group $G_2$. Hence, using the enumeration of \cite[Sections 5.2, 6.2]{pelsch} and in particular the notation of equations (6.3), (6.5), (6.6) and (6.7) of \cite{pelsch}, we arrive at the regular complexes\\[-.25in]
\begin{equation}
\label{thekones}
\begin{array}{rl}
\K_1(1, 1) \;:=&\!\! \K_3(1, 2)^{\,\lambda_1 (\hat{R}_2 \widetilde{R}_2 \hat{R}_2)},\\
\K_2(1, 1) \;:=&\!\! \K_5(1, 2)^{\,\lambda_1 (\widetilde{R}_2)},\\
\K_3(1, 1) \;:=&\!\! \K_6(1, 2)^{\,\lambda_1 (R_2 \hat{R}_2 R_2)},\\
\K_4(1, 1) \;:=&\!\! \K_7(1, 2)^{\,\lambda_1 (\hat{R}_2)},
\end{array}
\end{equation}
all with mirror vector $(1,1)$ and with special group $[3, 4]$. (Recall our convention to label regular complexes with their mirror vectors.)  Each new complex has the same twin vertex and the same vertex-figure group (although with new generators) as the original complex, so its edge graph (and in particular, its vertex-set) must be the same as that of the original complex. (Recall from \cite{pelsch} that the vertex $v$ of the base edge distinct from the base vertex is called the {\em twin vertex\/} of the complex.)  Similarly, since the element of $G$ that defines $\lambda_1$ belongs to $G_2$ and hence stabilizes the base vertex of the vertex-figure at $o$, the vertex-figure itself remains unchanged under $\lambda_1$, so that the two complexes always have the same vertex-figure at~$o$. Moreover, the (dihedral) subgroup $G_2$ and hence the parameter $r$ remain the same under the operation. The finer geometry of these complexes can be described as follows.

The vertex-set of $\K_1(1, 1)$ is $\Lambda_{(a, a, a)}$. The faces are helices over triangles and their axes are parallel to the diagonals of $C$. There are six helical faces around each edge, permuted under a dihedral group $G_2=D_3$; thus $r=6$. Note that every edge $e$ of $\K_1(1, 1)$ is a main diagonal of a cube $C_e$ in the cubical tessellation with vertex-set $a\mathbb{Z}^3$. With this in mind, for any three, but no four, consecutive edges, $e,f,g$ (say), of any helical face, the three corresponding cubes $C_e$, $C_f$ and $C_g$ share an edge whose vertices are not vertices of this helical face. Moreover, for any four consecutive edges $e,f,g,h$ of a helical face, the two edges shared by $C_e$, $C_f$, $C_g$ and $C_f$, $C_g$, $C_h$, respectively, are adjacent edges (of a square face) of $C_f$. Each of the six helical faces of $\K_1(1,1)$ around an edge $e$ with vertices $u,v$ is now determined by one of the six edges of $C_e$ that do not contain $u$ or $v$. The vertex-figure of $\K_1(1, 1)$ at $o$ coincides with the vertex-figure of $\K_3(1, 2)$ at $o$, and hence is the double-edge graph of the cube with vertices $(\pm a, \pm a, \pm a)$. The vertex-figure group is $[3, 4]$.

The vertex-set of $\K_2(1, 1)$ is $a\mathbb{Z}^3 \setminus ((0, 0, a) + \Lambda_{(a, a, a)})$ and the faces again are helices over triangles with their axes parallel to the diagonals of $C$. Now the faces are those Petrie polygons of the cubical tessellation of $\E$ with vertex-set $a\mathbb{Z}^3$ that have no vertex in $(0, 0, a) + \Lambda_{(a, a, a)}$; thus any two, but no three, consecutive edges belong to the same square face, and any three, but no four, consecutive edges belong to the same cubical tile, of the cubical tessellation.  There are four helical faces around an edge of $\K_2(1,1)$, so $r=4$ (and $G_2=D_2$). As for the original complex $\K_5(1, 2)$, the vertex-figure of $\K_2(1, 1)$ at $o$ is the (planar) double-edge graph of the square with vertices $(\pm a, 0, 0)$ and $(0, \pm a, 0)$, and the vertex-figure group is $[4, 2] \cong D_4 \times C_2$.

The vertex-set of $\K_3(1, 1)$ is $a\mathbb{Z}^3$ and the faces again are helices over triangles with their axes parallel to the diagonals of $C$. Now the faces are all the Petrie polygons of the cubical tessellation of $\E$ with vertex-set $a\mathbb{Z}^3$, so $\K_3(1,1)$ contains $\K_2(1, 1)$ as a subcomplex. There are eight helical faces around an edge, so $r=8$ (and $G_2=D_4$). The vertex-figure of $\K_3(1, 1)$ at $o$ is the double-edge graph of the octahedron with vertices $(\pm a, 0, 0), (0, \pm a, 0), (0, 0, \pm a)$, and the vertex-figure group is $[3, 4]$.

The vertex-set of $\K_4(1, 1)$ is $\Lambda_{(2a, 2a, 0)} \cup ((a, -a, a) + \Lambda_{(2a, 2a, 0)})$, and the edges are main diagonals of cubes of the cubical tessellation of $\E$ with vertex-set $a \mathbb{Z}^3$. Now the faces are helices over squares with their axes parallel to the coordinate axes; in particular, the axis of the base face $F_2$ is parallel to the $y$-axis and the projection of $F_2$ along its axis onto the $xz$-plane is the square with vertices $(0, 0, 0)$, $(a, 0, a)$, $(0, 0, 2a)$ and $(-a, 0, a)$. Each edge belongs to six helical faces (that is, $r=6$ and $G_2=D_3$), and these have the property that each coordinate direction of $\E$ occurs exactly twice among the directions of their axes. The vertex-figure of $\K_4(1, 1)$ at $o$ is the double-edge graph of the tetrahedron with vertices $(a, -a, a)$, $(-a, a, a)$, $(a, a, -a)$, $(-a, -a, -a)$. The vertex-figure group is $[3, 3]$. Note that the common edge graph of $\K_4(1,1)$ and $\K_7(1,2)$ is the famous {\em diamond net\/} modeling the diamond crystal (see \cite{pelsch}, as well as \cite[p. 241]{arp} and \cite[pp. 117,118]{wells}).

\subsection{The five complexes not derived through $\lambda_1$}
\label{notlam}

Next we enumerate the regular complexes with mirror vector $(1,1)$ for which either $G_2$ is cyclic, or $G_2$ is dihedral and the axis of $R_1$ is not contained in a mirror of a plane reflection in $G_2$. Now we cannot apply any of the operations $\lambda_0$ and $\lambda_1$ but instead must deal with the geometry directly. Recall that $L_3$ and $L_{3}'$ denote the axes of $R_0 R_1$ and $R_0'R_1$, respectively. We break our discussion into three cases, I, II and III respectively, according as $L_3'$ is a coordinate axis, $L_3'$ is parallel to a face diagonal of $C$, or $L_3'$ is parallel to a main diagonal of $C$. (Recall here that $G_*=[3,4]^+$ or $[3,4]$.) In each case there is just one choice for $L_{3}'$ (up to conjugacy), namely the line through $o$ with direction vector $(1,0,0)$, $(0, 1, 1)$, or $(1, 1, 1)$ (say), respectively.

\medskip
\noindent{\bf Case I: $L_3'$ is a coordinate axis}
\medskip

Suppose $R_0'R_1$ is a rotation whose axis $L_3'$ is the $x$-axis. Then there are two possible choices for each of the rotation axes $L_0'$ of $R_0'$ and $L_1$ of $R_1$ (perpendicular to $L_3'$), namely a coordinate axis or a line through the midpoints of a pair of antipodal edges of $C$.  If $L_0'$ and $L_1$ are perpendicular, then $R_0' R_1$ must be a half-turn and the faces of $\K$ must be zigzags. However, as we shall see, this case will not actually occur under our assumptions. On the other hand, if $L_0'$ and $L_1$ are inclined at an angle $\pi/4$, then $R_0' R_1$ is a $4$-fold rotation and the faces of $\K$ are helices over squares.

\medskip
\noindent{\em Case Ia: $L_0'$ and $L_1$ both are coordinate axes}
\medskip

We can rule out this possibility on the following grounds. Suppose $L_0'$ is the $y$-axis and $L_1$ is the $z$-axis. Since the rotation axis $L_2$ of $S$ must be orthogonal to $L_0'$ but not to $L_1$, the only possible choice for $L_2$ is the line through $o$ and $(1, 0, 1)$. This immediately implies that $S$ is a half-turn and that $G_2$ is the dihedral group generated by the reflections in the $xz$-plane and the plane $x=z$. (Bear in mind that $r\geq 3$.) However, the $xz$-plane is invariant under $R_0'$, $R_1$ and $G_2$, and hence under all of $G_*$, contradicting our assumption of irreducibility of $G$. Therefore this case cannot occur.

\medskip
\noindent{\em Case Ib: $L_0'$ is a coordinate axis and $L_1$ is parallel to a face diagonal of $C$}
\medskip

Suppose $L_0'$ is the $y$-axis and $L_1$ is the line through $o$ and $(0, 1, 1)$. Now there are two possibilities for the rotation axis $L_2$ for $S$, namely the $z$-axis or the line through $o$ and $(1, 0, 1)$.

We first eliminate the possibility that $L_2$ is the $z$-axis. When $L_2$ is the $z$-axis, the group $G_2$ can be cyclic of order $4$ or dihedral. We first rule out the latter possibility as follows. Bear in mind that $L_1$ does not lie in the mirror of a plane reflection in $G_2$. Now if $G_2$ is dihedral, then it cannot contain the reflection in the $yz$-plane and must necessarily have order $4$ and be generated by the reflections in the planes $x=y$ and $x=-y$. However, then Lemma~\ref{cubelemma} shows that the generators $R_1$ and $G_2$ of the vertex-figure group must already generate the full special group, $[3, 4]$, implying that $G_2$ must actually have order $8$, contradicting our earlier claim. Thus $G_2$ cannot be dihedral.

Next we consider the possibility that $G_2$ is cyclic of order $4$ (and $L_2$ is the $z$-axis). Since $L_0'$, $L_1$ and $L_2$ (and hence $F_1$) are coplanar, $L_0$ and $L_1$ are also coplanar and intersect at an angle $\pi/4$. Therefore the base face $F_2$ of $\K$ must be a planar square. Since now all generators of $G$ are rotations, $G$ consists only of proper (orientation preserving) isometries. This suggests that $G$ is the even subgroup (of all proper isometries) of the symmetry group of the cubical tessellation in $\E$. This can indeed be verified by the following argument (or alternatively by Wythoff's construction). Let $T_0$, $T_1$, $T_2$, $T_3$ denote the distinguished plane reflections generating the symmetry group $[4,3,4]$ of the cubical tessellation $\{4,3,4\}$ of $\E$, chosen in such a way that $T_0 T_3$, $T_1 T_3$ and $T_2 T_3$ coincide with $R_0$, $R_1$ and $S$, respectively. Since these three elements generate the rotation subgroup of $[4,3,4]$ we conclude that $\K$ would have to coincide with the $2$-skeleton of $\{4, 3, 4\}$, which is impossible as $\K$ is simply flag-transitive. Thus $G_2$ cannot be cyclic of order $4$, completing our argument that in Case Ib the rotation axis $L_2$ of $S$ cannot be the $z$-axis.

We now analyze the case when $L_2$ is the line through $o$ and $(1, 0, 1)$. In this case the twin vertex has the form $(a, 0, a)$ for some $a \ne 0$. Since $r\geq 3$, the group $G_2$ must necessarily be dihedral of order $4$, generated by the reflections $R_2$ in the plane $x=z$ and $\hat{R}_2$ in the $xz$-plane. Then $G$ has generators $R_0$, $R_1$, $R_2$ and $\hat{R}_2$ given by
\begin{equation}
\label{genk511}
\begin{array}{rccl}
R_0\colon    & (x,y,z)  &\mapsto & (-x,y,-z) + (a,0,a),\\
R_{1}\colon & (x,y,z)  &\mapsto  & (-x,z,y),\\
R_2\colon  & (x,y,z) &\mapsto & (z,y,x),\\
\hat{R}_2\colon  & (x,y,z) &\mapsto & (x,-y,z),
\end{array}
\end{equation}
with $a \ne 0$ (see Figure~\ref{k511}). This determines a new regular complex, denoted $\K_5(1, 1)$, with faces given by helices over squares and with four faces around each edge (that is, $r=4$ and $G_2=D_2$).

The vertex-set of $\K_5(1, 1)$ is $\Lambda_{(a, a, 0)}$. The helical faces have their axes parallel to a coordinate axis, and each coordinate axis occurs. The set of faces of $\K_5(1, 1)$ splits into three classes each determined by the  coordinate axis that specifies the direction for the axes of its members. The faces in each class constitute four copies of the (blended) apeirohedron $\{4,4\} \#\{\infty\}$ (see \cite[p. 222]{arp}). From any copy in the class determined by the $x$-direction we can obtain another copy through translation by $(2a,0,0)$; the remaining two copies then are  obtained by rotating the first two copies by $\pi/2$ about the axis of a helical face. The situation is similar for the other two classes. Thus $\K_5(1, 1)$ is a regular complex that can be viewed as a compound of twelve such apeirohedra, four for each coordinate direction. The vertex-figure group of $\K_5(1,1)$ is the full octahedral group (see Lemma~\ref{cubelemma}), and the vertex-figures are isomorphic to the edge graph of the cuboctahedron.

Observe here that the vertex-figure of $\K_5(1,1)$ at $o$ induces a non-standard realization of the cuboctahedron with equilateral triangular and skew square faces, and with vertices $(\pm 1, \pm 1, 0)$, $(\pm 1, 0, \pm 1)$ and $(0, \pm 1, \pm 1)$ (say). The four vertices adjacent to $(1, 1, 0)$ are $(-1, 0, 1)$, $(0, -1, 1)$, $(-1, 0, -1)$ and $(0, -1, -1)$; these correspond to the midpoints of the four edges of $C$ sharing a vertex with the edge of $C$ opposite the edge with midpoint $(1, 1, 0)$. A typical triangle has vertices $(1, 1, 0)$, $(-1, 0, 1)$ and $(0, -1, -1)$, while a typical square is given by the vertices $(1, 1, 0)$, $(-1, 0, 1)$, $(1, -1, 0)$ and $(-1, 0, -1)$, in that order.

These observations also shed some light on why four copies of $\{4,4\} \#\{\infty\}$ per coordinate direction are needed to cover all helical faces of $\K_5(1,1)$ with this direction. In fact, a single copy of this apeirohedron accounts for just one square of the cuboctahedral vertex-figure, so a pair of opposite squares requires two such copies; on the other hand, the base vertex $o$ lies in just one half of the helical faces of $\K_5(1,1)$ with a given direction, with the other half accounting for the two additional copies of the apeirohedron.

\begin{figure}
\begin{center}
\includegraphics[width=8cm, height=7cm]{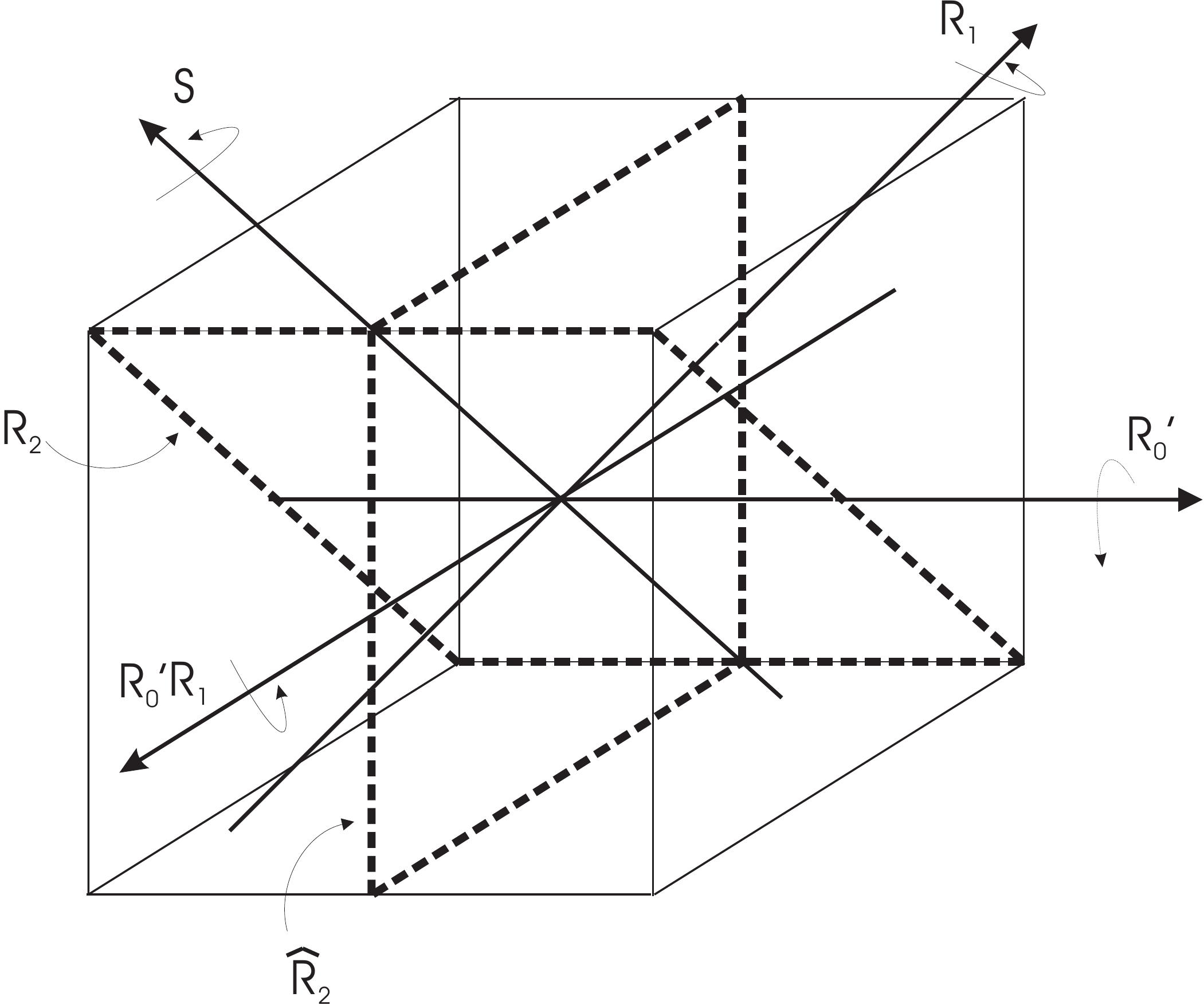}
\caption{The special group of the complex $\K_{5}(1, 1)$}\label{k511}
\end{center}
\end{figure}

\medskip
\noindent{\em Case Ic: $L_0'$ is parallel to a face diagonal of $C$ and $L_1$ is a coordinate axis}
\medskip

Suppose $L_0'$ is the line through $o$ and $(0, 1, 1)$, and $L_1$ is the $y$-axis. Then there are two possible choices for $L_2$, namely the line through $o$ and $(0, 1, -1)$, or the line through $o$ and $(1, 1, -1)$. However, if $L_2$ is the line through $o$ and $(0, 1, -1)$, then $S$ is a half-turn and $G_2$ is the dihedral group generated by the reflections in the $yz$-plane and the plane $y=-z$, which contradicts our previous hypothesis that the axis of $R_1$ not be contained in a mirror of a plane reflection in $G_2$. Therefore we may assume that $L_2$ is the line through $o$ and $(1, 1, -1)$ and hence that the twin vertex has the form $(a, a, -a)$ with $a \ne 0$.

Now $S$ is a $3$-fold rotation and the subgroup $G_2$ must be cyclic of order $3$. In fact, the axis $L_1$ is contained in the plane $x=-z$, which would become the mirror of a plane reflection if $G_2$ was dihedral of order $6$. It follows that $G$ has generators $R_0$, $R_1$ and $S$ given by
\begin{equation}
\begin{array}{rccl}
R_0\colon    & (x,y,z)  &\mapsto & (-x,z,y) + (a,a,-a),\\
R_{1}\colon & (x,y,z)  &\mapsto  & (-x,y,-z),\\
S\colon  & (x,y,z) &\mapsto & (y,-z,-x),
\end{array}
\end{equation}
with $a \ne 0$ (see Figure~\ref{k611}). These generators yield a regular complex, denoted $\K_6(1, 1)$, which again has faces given by helices over squares but now with three faces surrounding each edge (that is, $r=3$ and $G_2=C_3$).

\begin{figure}
\begin{center}
\includegraphics[width=8cm, height=7cm]{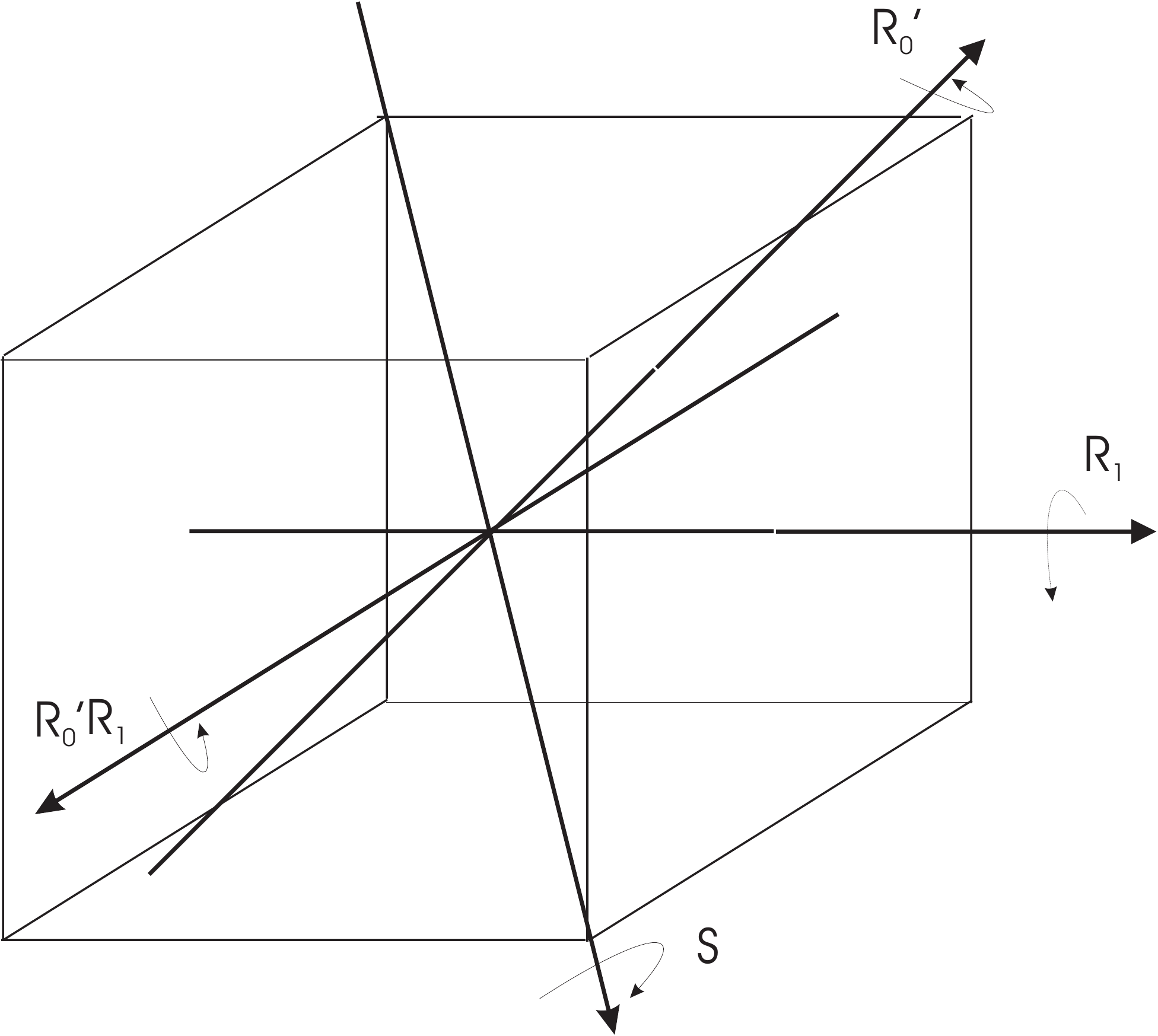}
\caption{The special group of the complex $\K_{6}(1, 1)$}\label{k611}
\end{center}
\end{figure}

The vertex- and edge-sets of $\K_6(1, 1)$ coincide with those of $\K_4(1, 1)$, respectively; in particular, the edge graphs of $\K_6(1, 1)$ and $\K_4(1, 1)$ are the same and form a diamond net. In fact, $\K_6(1, 1)$ is a subcomplex of $\K_4(1, 1)$ made up of only half the faces of the latter. The group $G$ consists of all proper isometries in the group of $\K_4(1, 1)$, and $G_2$ is the cyclic subgroup of the corresponding (dihedral) group of $\K_4(1,1)$. Thus, in a sense, the faces of $\K_6(1, 1)$ are exactly the right-handed (say) helices of $\K_4(1, 1)$. Now each coordinate axis is parallel to the axis of just one helix containing a given edge of $\K_6(1, 1)$. The vertex-figure of $\K_6(1, 1)$ at $o$ is the (simple) edge-graph of the tetrahedron with vertices $(a, -a, a)$, $(-a, a, a)$, $(a, a, -a)$, $(-a, -a, -a)$. The vertex-figure group is $[3, 3]^+$.

\medskip
\noindent{\em Case Id: $L_0'$ and $L_1$ both are parallel to face diagonals of $C$}
\medskip

We show that this case cannot contribute a regular complex. Suppose $L_0'$ is the line through $o$ and $(0, 1, 1)$, and $L_1$ is the line through $o$ and $(0, 1, -1)$. Now there is just one choice for $L_2$, namely the line through $o$ and $(-1, 1, -1)$, and then $S$ is a $3$-fold rotation. Moreover,  $G_2$ must be cyclic of order $3$. In fact, if $G_2$ was dihedral, the plane $y=-z$ would become the mirror of a reflection in $G_2$ and contain the axis $L_1$ of $R_1$, contrary to our previous hypothesis on $G_2$ and $L_1$.

Since the lines $L_0'$ and $L_1$ are perpendicular, the faces are planar zigzags. All three generators $R_0$, $R_1$ and $S$ of $G$ are again proper isometries. We claim that now $\K$ must be the $2$-skeleton of the regular $4$-apeirotope
\[\mathcal{P} := \{\{\infty, 4\}_4 \# \{\}, \{4, 3\}\}\]
in $\E$; however, this is impossible as $\K$ is simply flag-transitive.  In fact, let $T_0, T_1, T_2, T_3$ denote the distinguished generators of the symmetry group $G(\mathcal{P})$ of $\mathcal{P}$, where $T_0$ is the point reflection in $\frac{1}{2}v$, with $v = (-a, a, -a)$, and the distinguished generators $T_1, T_2, T_3$ for the cube $\{4,3\}$ (the vertex-figure of $\mathcal{P}$) are chosen in such a way that $T_0 T_3 = R_0$, $T_1 T_3 = R_1$ and $T_2 T_3 = S$. Since these three rotations generate the even subgroup of $G(\mathcal{P})$, it follows that $\K$ must necessarily be the $2$-skeleton of $\mathcal{P}$. Thus Case Id does not yield a (simply flag-transitive) regular complex.

\medskip
\noindent{\bf Case II: $L_3'$ is parallel to a face diagonal of $C$}
\medskip

We shall see that Case II does not contribute a regular complex (with a simply flag-transitive group). Suppose $R_0' R_1$ is a rotation whose axis $L_3'$ is the line passing through the midpoints of a pair of antipodal edges of $C$, the line through $o$ and $(0, 1, 1)$ (say). Then $R_0' R_1$ must be a half-turn and the faces of $\K$ must be zigzags.
There there are two possible choices for the axis $L_0'$ of the half-turn $R_0'$, namely the $x$-axis or the line through $o$ and $(0,1,-1)$. In each case $L_1$ must necessarily be perpendicular to $L_0$.

\medskip
\noindent{\em Case IIa: $L_0'$ is the $x$-axis}
\medskip

If $L_0'$ is the $x$-axis, then $L_1$ must necessarily be the line through $o$ and $(0, 1, -1)$. In this situation, the rotation axis $L_2$ of $S$ must be a coordinate axis, the $y$-axis (say), and the subgroup $G_2$ must be cyclic of order $4$ or dihedral of order $4$ or $8$. We can rule out the possibility that $G_2$ is dihederal. In fact, if $G_2$ was dihedral, then since the $yz$-plane contains $L_1$, the group $G_2$ would necessarily have order $4$ and be generated by the reflections in the planes $x = \pm z$; however, this would immediately force the vertex-figure group to be the full special group $[3, 4]$ and then the subgroup $G_2$ to have order $8$ (see Lemma~\ref{cubelemma}).

Therefore we may assume that $G_2$ is cyclic of order $4$. Since then $G$ is generated by rotations and $\K$ has planar faces, we can proceed as in Case Id and establish that $\K$ must be the $2$-skeleton of the regular $4$-apeirotope
\[\mathcal{P} := \{\{\infty, 3\}_6 \# \{\}, \{3, 4\}\}.\]
In fact, the distinguished generators $T_0, T_1, T_2, T_3$ of $G(\mathcal{P})$ can once again be chosen in such a way that $T_0 T_3 = R_0$, $T_1 T_3 = R_1$ and $T_2 T_3 = S$. Hence $\K$ must be the $2$-skeleton of $\mathcal P$, which we know to be impossible.

\medskip
\noindent{\em Case IIb: $L_0'$ is the line through $o$ and $(0, 1, -1)$}
\medskip

If $L_0'$ is the line through $o$ and $(0, 1, -1)$, then $L_1$ and $L_2$ must necessarily be the $x$-axis and the line through $o$ and $(1, 1, +1)$ respectively. If $G_2$ was dihedral, then, contrary to our earlier hypothesis, the plane $y = z$ would become the mirror of a reflection in $G_2$ containing $L_1$. Hence $G_2$ must be cyclic of order $3$, so again $G$ is generated by rotations. Now $\K$ must be the $2$-skeleton of the regular $4$-apeirotope
\[\mathcal{P} := \{\{\infty, 3\}_6 \# \{\}, \{3, 3\}\},\]
once again by the same arguments involving a choice of generators of $G(\mathcal{P})$.

\medskip
\noindent{\bf Case III: $L_3'$ is the line through a main diagonal of $C$}
\medskip

Suppose $R_0' R_1$ is a rotation whose axis $L_3'$ is the line through a main diagonal of $C$, the line through the vertices $\pm (1, 1, 1)$ (say). Then we may assume that $R_0'$ is the half-turn about the line $L_{0}'$ through $o$ and $(1, -1, 0)$, and that $R_1$ is the half-turn about the line $L_1$ through $o$ and $(1, 0, -1)$. It follows that $R_0' R_1$ is a $3$-fold rotation and that the faces of $\K$ are helices over triangles. We now have three choices for the axis $L_2$ of $S$, namely the $z$-axis, the line through $o$ and $(1, 1, 0)$, or the line through $o$ and $(1, 1, -1)$.

\medskip
\noindent{\em Case IIIa: $L_2$ is the $z$-axis}
\medskip

If $L_2$ is the $z$-axis, then the twin vertex has the form $(0, 0, a)$ for some $a \ne 0$. We already discussed the case, ruled out here by our previous assumptions, that $G_2$ is a dihedral group with the $xz$-plane as a reflection mirror that contains $L_1$; this gave us the complexes $\K_2(1, 1)$ 
when $G_2$ was dihedral of order $4$, and $\K_3(1, 1)$ when $G_2$ was dihedral of order $8$. There is just one other way for $G_2$ to be dihedral, and this can be eliminated as follows. It occurs when $G_2$ is generated by the reflections in the planes $x = \pm y$ and hence is of order $4$. However, then the vertex-figure group is the full special group $[3, 4]$,  forcing $G_2$ to have order $8$ rather than $4$ (see again Lemma~\ref{cubelemma}).

This only leaves the possibility that $G_2$ is cyclic of order $4$. Then the generators $R_0$, $R_1$ and $S$ of $G$ are given by
\begin{equation}
\begin{array}{rccl}
R_0\colon    & (x,y,z)  &\mapsto & (-y, -x, -z) + (0,0,a),\\
R_{1}\colon & (x,y,z)  &\mapsto  & (-z, -y, -x),\\
S\colon  & (x,y,z) &\mapsto & (-y, x, z),
\end{array}
\end{equation}
for some $a \ne 0$ (see Figure~\ref{k711}). Now we obtain a regular complex, denoted $\K_7(1, 1)$, with helices over triangles as faces, four surrounding each edge (that is, $r=4$ and $G_{2}=C_4$). 

\begin{figure}
\begin{center}
\includegraphics[width=8cm, height=7cm]{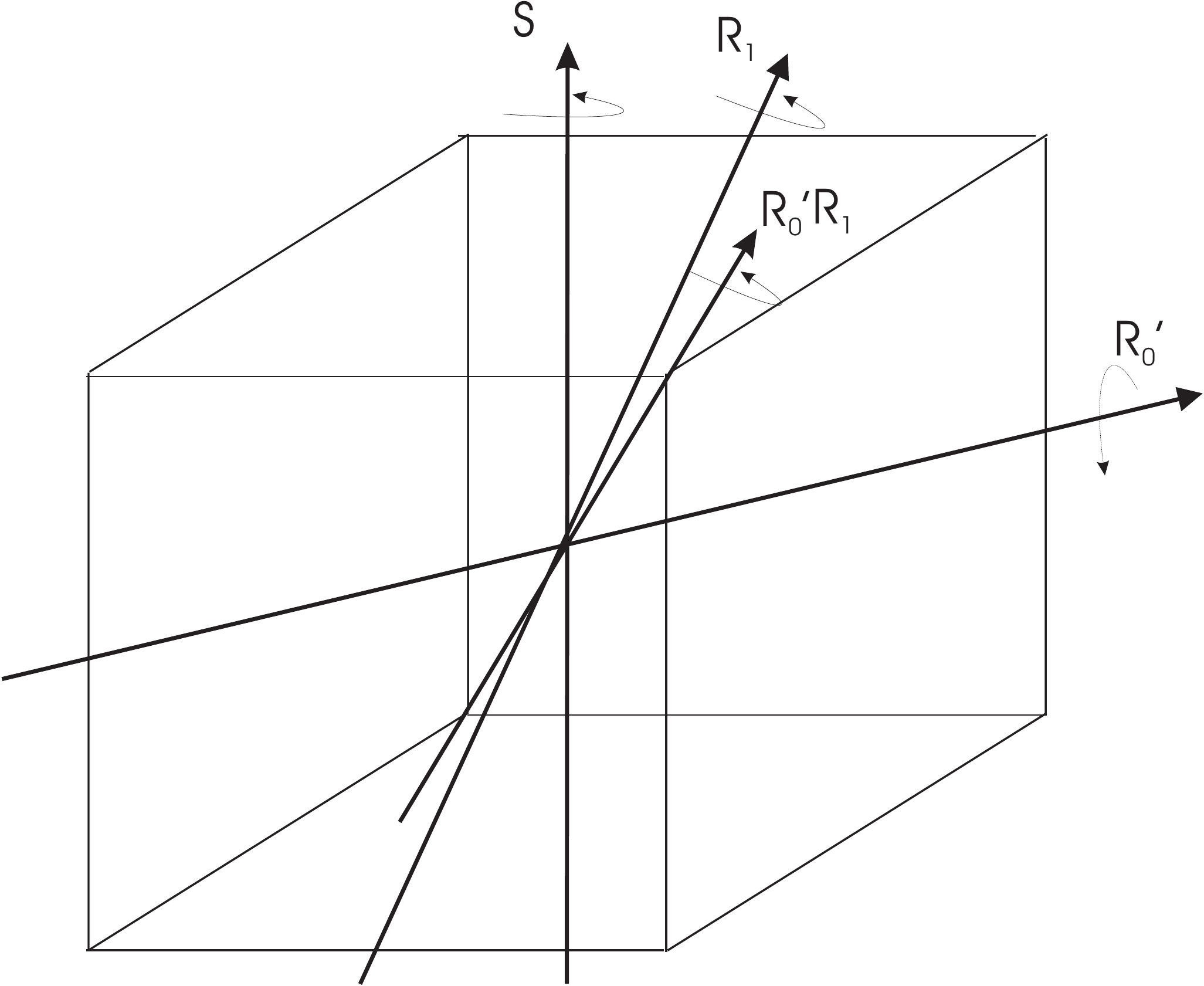}
\caption{The special group of the complex $\K_{7}(1, 1)$}\label{k711}
\end{center}
\end{figure}

The edge graph of $\K_7(1, 1)$ coincides with the edge graph of $\K_3(1, 1)$. The group of $\K_7(1,1)$ consists of all proper isometries in the group of $\K_3(1, 1)$, and its subgroup $G_2$ is just the cyclic subgroup of the corresponding group for $\K_3(1,1)$. Hence the faces of $\K_7(1, 1)$ are just the right-handed (say) Petrie polygons of the cubical tessellation of $\E$ with vertex-set $a\mathbb{Z}^3$. The vertex-figure of $\K_7(1, 1)$ at $o$ is the (simple) edge graph of the octahedron with vertices $(\pm a, 0, 0), (0, \pm a, 0), (0, 0, \pm a)$. The vertex-figure group is $[3, 4]^+$.

\medskip
\noindent{\em Case IIIb: $L_2$ is the line through $o$ and $(1, 1, 0)$}
\medskip

If $L_2$ is the line through $o$ and $(1, 1, 0)$, the twin vertex has the form $(a, a, 0)$ for some $a \ne 0$. The group  $G_2$ must be dihedral of order $4$ (recall that $r\geq 3$), generated by the reflections $R_2$ and $\hat{R}_2$ in the $xy$-plane and the plane $x=y$, respectively. Then $G$ is generated by $R_0$, $R_1$, $R_2$ and $\hat{R}_2$ given by
\begin{equation}
\begin{array}{rccl}
R_0\colon    & (x,y,z)  &\mapsto & (-y, -x, -z) + (a, a, 0),\\
R_{1}\colon & (x,y,z)  &\mapsto  & (-z, -y, -x),\\
R_2\colon  & (x,y,z) &\mapsto & (x, y, -z),\\
\hat{R}_2\colon  & (x,y,z) &\mapsto & (y, x, z),
\end{array}
\end{equation}
for some $a \ne 0$ (see Figure~\ref{k811}). This leads to a new regular complex, denoted $\K_8(1, 1)$, which has
helices over triangles as faces such that four surround each edge (that is, $r=4$ and $G_2=D_2$). 

\begin{figure}
\begin{center}
\includegraphics[width=8cm, height=7cm]{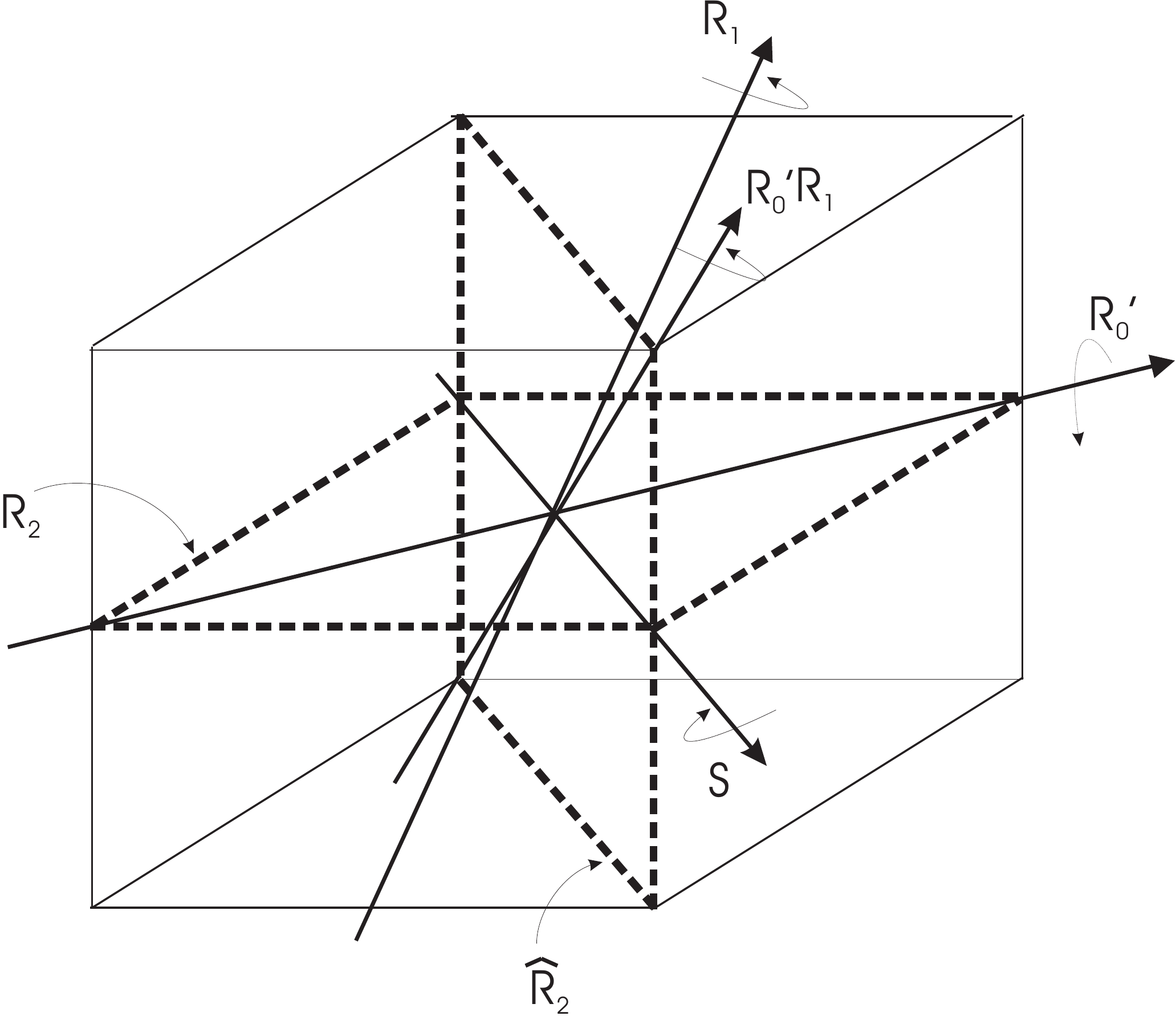}
\caption{The special group of the complex $\K_{8}(1, 1)$}\label{k811}
\end{center}
\end{figure}

The vertex-set of $\K_8(1, 1)$ is $\Lambda_{(a, a, 0)}$. Now the edges are face diagonals of square faces of the cubical tessellation of $\E$ with vertex-set $a \mathbb{Z}^3$.  Any three consecutive edges $e_1$, $e_2$ and $e_3$ of a face of $\K_8(1, 1)$ can be seen to lie in a $2a\!\times\! 2a\!\times\! 2a$ cube $Q$ formed from eight cubes of this tessellation. The middle edge $e_2$ joins the midpoints of two adjacent faces $f_1$ and $f_2$ of $Q$, while the first edge $e_1$ and the last edge $e_3$, respectively, lie in $f_1$ and $f_2$ and join the midpoints of $f_1$ and $f_2$ to opposite vertices $u$ and $w$ of $Q$. To construct the entire
helical face we only need to translate $e_1, e_2, e_3$ by integral multiples of the vector $u-w$. The vertex-figure of
$\K_8(1, 1)$ is isomorphic to the (simple) edge graph of the cuboctahedron (inducing the same non-standard realization of the cuboctahedron as for the vertex-figure of $\K_5(1, 1)$). The vertex-figure group is $[3, 4]$.

\medskip
\noindent{\em Case IIIc: $L_2$ is the line through $o$ and $(1, 1, -1)$}
\medskip

Finally, if $L_2$ is the line through $o$ and $(1, 1, -1)$, then the twin vertex has the form $(a, a, -a)$ for some $a \ne 0$. Now observe that the plane $x=-z$ contains both $L_1$ and $L_2$. It follows that, if $G_2$ was dihedral of order $6$, then $L_1$ would necessarily lie in the mirror of a plane reflection of $G_2$; in fact, this possibility just yielded the complex $\K_1(1, 1)$ that we described earlier. On the other hand, if $G_2$ is cyclic of order $3$, then $G$ is generated by the proper isometries
\begin{equation}
\begin{array}{rccl}
R_0\colon    & (x,y,z)  &\mapsto & (-y, -x, -z) + (a, a, 0),\\
R_{1}\colon & (x,y,z)  &\mapsto  & (-z, -y, -x),\\
S\colon  & (x,y,z) &\mapsto & (-z, x, -y),
\end{array}
\end{equation}
for some $a \ne 0$ (see~Figure \ref{k911}). The resulting regular complex, denoted $\K_9(1, 1)$, has helices over triangles as faces such that three surround each edge (that is, $r=3$ and $G_2=C_3$). 

\begin{figure}
\begin{center}
\includegraphics[width=8cm, height=7cm]{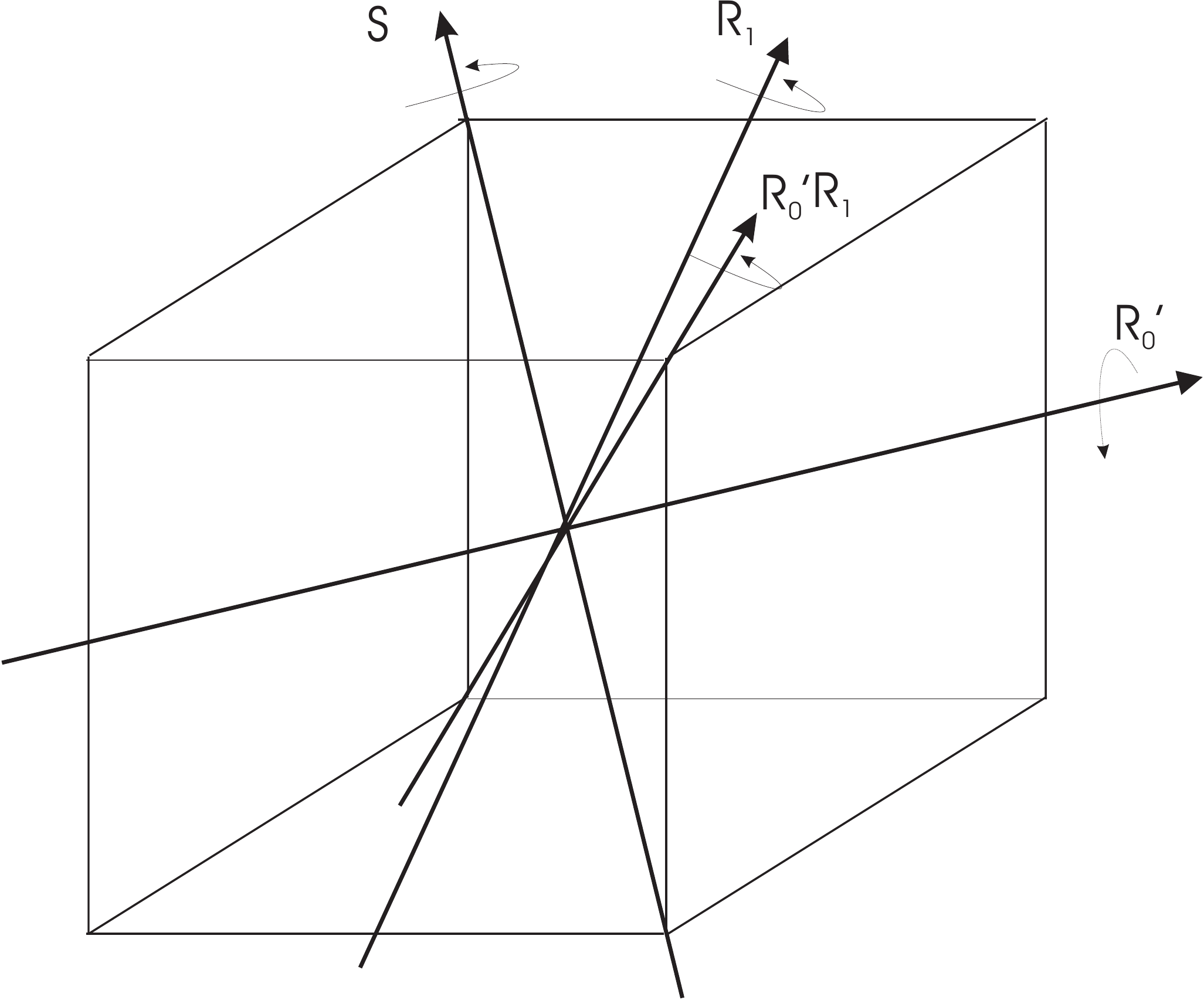}
\caption{The special group of the complex $\K_{9}(1, 1)$}\label{k911}
\end{center}
\end{figure}

Just like $\K_6(1, 1)$ and $\K_7(1, 1)$, this new regular complex $\K_9(1, 1)$ has only right-handed (say) helices as faces. Now the faces are given by all the right-handed helical faces of $\K_1(1, 1)$. In particular, the edge graphs of $\K_9(1, 1)$ and $\K_1(1, 1)$ are the same. The group of $\K_9(1, 1)$ consists of the proper isometries in the group of $\K_1(1, 1)$, and its subgroup $G_2$ is just the cyclic subgroup of the corresponding group for $K_1(1,1)$. The vertex-figure is the (simple) edge graph of the cube with vertices $(\pm a, \pm a, \pm a)$ and the vertex-figure group is $[3, 4]^+$.
\bigskip

In conclusion, our discussion in Sections~\ref{lam} and \ref{notlam} establishes the following theorem.

\begin{theorem}
Apart from polyhedra, the complexes $\K_1(1, 1), \dots, \K_9(1, 1)$ described in this section are the only simply flag-transitive regular polygonal complexes with mirror vector $(1, 1)$.
\end{theorem}

\section{Complexes with mirror vector~$(0,k)$ and dihedral $G_2$}
\label{mir0k}

We begin with the following lemma about regular complexes with face mirrors, which is also of independent interest.

\begin{lemma}
\label{halfturn}
Let $\mathcal{L}$ be a regular complex with face mirrors. Then the subgroups $G_0(\mathcal{L})$ and $G_{1}(\mathcal{L})$ of $G(\mathcal{L})$ each contain exactly one half-turn (with its axis contained in, or perpendicular to, the plane of the base face). Moreover, each either contains one pair of commuting plane reflections (with their mirrors given by the plane through the base face, and the perpendicular plane meeting the first in the axis of the half-turn), or one plane reflection (with its mirror given by the plane through the base face) and one point reflection (with center contained in the base face).
\end{lemma}

\begin{proof}
Recall from \cite[\S 3]{pelsch} that $\mathcal{L}$ has planar faces and that $G(\mathcal{L})$ has flag-stabilizers of order $2$. In particular, the stabilizer of the base flag is generated by the reflection $R$ in the plane through the base face. Now,  since $G_0(\mathcal{L})\cong C_{2}\times C_{2} \cong G_1(\mathcal{L})$ and $R$ is an improper isometry contained in $G_0(\mathcal{L})$ and $G_1(\mathcal{L})$, there is just one non-trivial proper isometry in each of $G_0(\mathcal{L})$ and $G_1(\mathcal{L})$. Thus each of these subgroups contains a unique half-turn.  The other two involutions in each subgroup commute and their product is this half-turn; then this only leaves the two possibilities described.  (An alternative proof of the lemma could be obtained from \cite[Theorem 4.1]{pelsch} and would provide more detailed information about $G_{0}(\mathcal{L})$ or $G_{1}(\mathcal{L})$.)
\end{proof}

Now according to Lemma~\ref{apptwo}, each simply flag-transitive complex $\K$ with mirror vector $(0,k)$ and a dihedral group $G_2$ can be obtained from a regular complex $\mathcal{L}$ which either has face mirrors or is simply flag-transitive with mirror vector $(1,k)$, by applying to $\mathcal{L}$ the operation $\lambda_0$ determined by a plane reflection from $G_2(\mathcal{L})$ with mirror perpendicular to the axis of the (unique) half-turn in $G_0(\mathcal{L})$; here necessarily $\mathcal{L}=\K^{\lambda_0}$.

We can first rule out the possibility that the complex $\mathcal{L}$ has face-mirrors. This follows from our next lemma applied with $\mathcal{L}=\K^{\lambda_0}$, noting that then $\mathcal{L}'=\K$ would also have face mirrors, which contradicts our assumptions.

\begin{lemma}\label{facemirlambda01}
Let $\mathcal{L}$ be a regular complex with face mirrors, and let $R_0$ be the half-turn in $G_0(\mathcal{L})$. Suppose $G_2(\mathcal{L})$ contains a plane reflection $R_2$ whose mirror is perpendicular to the axis of $R_0$. Let $\mathcal{L}'$ denote the regular complex obtained from $\mathcal{L}$ by the operation $\lambda_0(R_2)$ associated with $R_2$. Then $\mathcal{L}'$ also has face mirrors.
\end{lemma}

\begin{proof}
Let $L$ denote the plane through the base face $F_2$ of $\mathcal{L}$. Then the axis $L_1$ of the unique half-turn $R_1$ in $G_1(\mathcal{L})$ lies in $L$, since otherwise $F_2$ would be a linear apeirogon. Also, the mirrors of the two plane reflections in $G_1(\mathcal{L})$ are $L$ and the plane through $L_1$ perpendicular to $L$. In particular, $G_1(\mathcal{L})$ leaves $L$ invariant. Furthermore, the mirror of the point reflection $R_0 R_2$ is the midpoint of the base edge $F_1$, which is also contained in $L$. Since the symmetries of $\mathcal{L}$ are also symmetries of $\mathcal{L}'$, the reflection in $L$ is also a symmetry of $\mathcal{L}'$.

Now observe that the complex $\mathcal{L}'$ can be obtained from Wythoff's construction with the same initial vertex as for $\mathcal{L}$, namely the base vertex $F_0$ of $\mathcal{L}$, which also lies in $L$. Since the vertices of the base face $F_2'$ of $\mathcal{L}'$ are just the images of $F_0$ under the group generated by $R_0 R_2$ and $R_1$, the face $F_2'$ must entirely lie in $L$ and hence be planar. On the other hand, the reflection in $L$ is a symmetry of $\mathcal{L}'$. Thus $\mathcal{L}'$ also has face mirrors.
\end{proof}

Thus, in order to enumerate the simply flag-transitive complexes $\K$ with mirror vector $(0,k)$ and a dihedral subgroup $G_2$, it is sufficient to apply the operation $\lambda_0$ to the simply flag-transitive complexes $\mathcal{L}$ with mirror vector $(1,k)$ and with a dihedral subgroup $G_2(\mathcal{L})$ containing a plane reflection with mirror perpendicular to the axis of the half-turn $R_0$ in $G_0(\mathcal{L})$. However, when $\lambda_0$ is applied to a complex $\mathcal{L}$ of this kind, the resulting regular complex $\mathcal{L}^{\lambda_0}$ can actually have face-mirrors and hence be discarded for our present enumeration. The following lemma, applied with $\K=\mathcal{L}^{\lambda_0}$, describes a scenario when this will occur.

\begin{lemma}\label{facemirr0k}
Let $\K$ be a regular complex with a dihedral group $G_2$ such that $G_0$ contains a point reflection and $G_1$ contains a line or plane reflection fixing the mirror of a plane reflection in $G_2$. Then $\K$ has face mirrors.
\end{lemma}

\begin{proof}
The base face $F_2$ of $\K$ can be obtained by Wythoff's construction from the orbit of the base vertex $F_0$ under the subgroup generated by the point reflection $R_0$ in $G_0$ and the line or plane reflection $R_1$ in $G_1$ that fixes the mirror $L$ of a plane reflection in $G_2$. Since $R_0$ also fixes $L$, this subgroup must preserve $L$. Therefore $F_2$ must lie in $L$ and the reflection in $L$ must stabilize the base flag. Thus $\K$ has face mirrors.
\end{proof}

\subsection{Complexes with mirror vector~$(0,1)$ and dihedral $G_2$}
\label{k01}

We now appeal to our enumeration of the simply flag-transitive polygonal complexes $\K$ with mirror vector $(1,1)$ in Section~\ref{mirr11} to determine all simply flag-transitive regular polygonal complexes $\K$ with mirror vector $(0,1)$ and a dihedral subgroup $G_2$. This is the case $(0,k)$ for $k=1$.

First note that if $\mathcal{L}$ is a simply flag-transitive polygonal complex with mirror vector $(1,1)$ obtained from a simply flag-transitive complex with mirror vector $(1,2)$ by the operation $\lambda_1$ as in Section~\ref{lam}, then the axis of the half-turn $R_1$ for $\mathcal{L}$ must lie in the mirror of a plane reflection in $G_2(\mathcal{L})$, so in particular $R_1$ must leave this mirror invariant. But then Lemma~\ref{facemirr0k} implies that the regular complex obtained from any such  complex $\mathcal{L}$ by operation $\lambda_0$ must actually have face mirrors. Therefore, in enumerating simply flag-transitive complexes $\K$ with mirror vector $(0,1)$ we can restrict ourselves to applying $\lambda_0$ to those complexes $\mathcal{L}$ of Section~\ref{mirr11} that were not derived by operation $\lambda_1$, that is, the complexes $\mathcal{L}=\K_i(1,1)$ with $i \ge 5$ described in Section~\ref{notlam}.

Moreover, since the complexes $\K_i(1,1)$ with $i=6,7,9$ have a cyclic subgroup $G_2(\K_i(1,1))$, we need only consider the complexes $\K_5(1,1)$ and $\K_8(1,1)$, which have a corresponding dihedral subgroup isomorphic to $D_2$. However, the mirror arrangements of $R_1$, $R_2$ and $\hat{R}_2$ depicted in Figures~\ref{k511} and~\ref{k811} coincide (up to congruence), so these three generators for $\K_5(1,1)$ and $\K_8(1,1)$ are the same (up to conjugacy). Moreover, the mirror of $\hat{R}_2$ is perpendicular to the axis of the half-turn $R_0$ in both cases. Consequently, since the fourth generator for $\mathcal{L}^{\lambda_0}$ is just the point reflection in the midpoint of the base edge, we have $\K_5(1,1)^{\lambda_0} \cong \K_8(1,1)^{\lambda_0}$; that is, the two regular complexes $\K_5(1,1)^{\lambda_0}$ and $\K_8(1,1)^{\lambda_0}$ are the same (up to congruence).

Thus there is just one simply flag-transitive polygonal complex with mirror vector $(0,1)$ and a dihedral group $G_2$, namely
\begin{equation}
\K(0, 1) \,:=\,\K_5(1, 1)^{\,\lambda_0 (\hat{R}_2)},
\end{equation}
with the notation as in Figure~\ref{k511}.

Just as the original complex $\K_5(1,1)$, the complex $\K(0,1)$ has vertex-set $\Lambda_{(a,a,0)}$. Its edges are face diagonals of the standard cubical tessellation with vertex-set $a\mathbb{Z}^3$, and its faces are zigzags, four around each edge. The four faces that surround an edge occur in two pairs of co-planar zigzags. Using the notation of (\ref{genk511}) and Figure~\ref{k511} we observe that the symmetry $R_{0}\hat{R}_2\!\cdot\! R_1$ of the base face $F_2$ which ``shifts" the vertices of $F_2$ by one step along $F_2$, is a glide reflection whose square is the translation by $(2a,-a,a)$. In particular, the base face of $\K(0,1)$ is given by
\[ F_{2} = \{(-a,a,0), (0,0,0),(a,0,a)\} + \mathbb{Z}\!\cdot\!(2a,-a,a) , \]
where here $(-a,a,0)$ and $(a,0,a)$ are the two vertices of $F_2$ adjacent to the base vertex $(0,0,0)$. The vertex-figure of $\K(0,1)$ at $o$ coincides with the vertex-figure of $\K_5(1,1)$ at $o$, that is, with the edge-graph of  a non-standard cuboctahedron with skew square faces.

The complex $\K(0,1)$ is closely related to the semiregular tessellation $\mathcal{S}$ of $\E$ by regular tetrahedra and octahedra described in Section~\ref{terba}. In fact, the zigzag base face of $\K(0,1)$ lies in the plane $x+y-z=0$ and is a $2$-zigzag of the regular tessellation of this plane by triangles formed from faces of the $2$-skeleton of $\mathcal{S}$; each $2$-zigzag of this tessellation occurs as a face of $\K(0,1)$. Recall here that a $2$-{\em zigzag\/} is an edge-path which leaves a vertex at the second edge from the one by which it entered, but in the oppositely oriented sense at alternate vertices (see \cite[p.\! 196]{arp}). (The notion of a $2$-zigzag of a regular map is not to be confused with that of a zigzag face of a complex.)  More generally, each $2$-zigzag of a triangular tessellation induced by $\mathcal{S}$ on the affine hull of a triangular face of $\mathcal{S}$ is a face of $\K(0,1)$, and all faces of $\K(0,1)$ arise in this way. Note that, for any such induced  triangular tessellation, the faces of $\K(0,1)$ that are its $2$-zigzags form the faces in a compound of three regular maps each isomorphic to $\{\infty,3\}_6$; this map could also be  obtained from the triangular tessellation by applying, in succession (in any order), the Petrie operation and the second facetting operation of~\cite[p. 196]{arp}.

\subsection{Complexes with mirror vector~$(0,2)$ and dihedral $G_2$}
\label{k02}

Next we determine all simply flag-transitive regular polygonal complexes with mirror vector $(0,2)$ and a dihedral subgroup $G_2$, now appealing to the enumeration of the simply flag-transitive complexes with mirror vector $(1,2)$. From the classification of these complexes with mirror vector $(1,2)$ in \cite[Section 6.2]{pelsch} we know that the generator $R_1$ for the complex $\mathcal{L}=\K_i(1,2)$ with $i=3,5,6,7$ is a plane reflection that fixes the mirror of a plane reflection in $G_2$, namely the plane reflection (with mirror perpendicular to that of $R_1$) employed in Section~\ref{lam} to define the operation $\lambda_1$. Now, when $\lambda_0$ is applied to these complexes $\mathcal{L}$, the resulting complex meets the assumptions of Lemma~\ref{facemirr0k} and hence must necessarily have face mirrors. On the other hand, $\K_2(1,2)$ has a cyclic group $G_2(\K_2(1,2))$, so in particular $\lambda_0$ cannot even be applied. Hence, in enumerating the simply flag-transitive complexes $\K$ with mirror vector $(0,2)$ we need only consider the effect of $\lambda_0$ on the complexes $\K_1(1,2)$, $\K_4(1,2)$ and $\K_8(1,2)$.

Next we observe that the complex $\K_4(1,2)^{\lambda_0}$, with $\lambda_{0}:=\lambda_{0}(\widehat{R}_2)$ and $\widehat{R_2}$ as in \cite[eq.~(6.4)]{pelsch}, actually coincides with the $2$-skeleton of the regular $4$-apeirotope $\apeir \{3, 4\}$ in $\E$ and therefore has face mirrors. In this case $G(\K_4(1,2))$ acts simply flag transitively on the flags of the $2$-skeleton of $\apeir \{3,4\}$, and $G(\K_4(1,2)^{\lambda_0})$ is strictly larger than $G(\K_4(1,2))$ (see our discussion after Lemma~\ref{lambda1}). Thus we can exclude this possibility as well. This follows from arguments very similar to those described later in Section~\ref{mirr02cyc}, so we will not include any details here. It suffices to say that the reflection $T_3$ in the $xy$-plane normalizes the distinguished generating subgroups of the symmetry group $G(\K_4(1,2)^{\lambda_0})$ of the complex, and hence is itself a symmetry of the complex not contained in $G(\K_4(1,2)^{\lambda_0})$ but stabilizing the base flag (lying in the $xy$-plane).

We further note that the mirror configurations of $R_1$, $R_2$ and $\hat{R}_2$ shown in Figures 2 and 9 of \cite{pelsch} are the same (up to congruence) so that $\K_1(1,2)^{\lambda_0} \cong \K_8(1,2)^{\lambda_0}$, again with $\lambda_{0}=\lambda_{0}(\hat{R}_2)$ in both cases. Hence, as in the previous subsection there is just one simply flag transitive regular polygonal complex with mirror vector $(0,2)$, namely
\begin{equation}
\K(0, 2) \, :=\,\K_1(1, 2)^{\,\lambda_0 (\hat{R}_2)},
\end{equation}
with the notation as in Section~6.2 of \cite{pelsch}.

The vertex-set of $\K(0,2)$ is also $\Lambda_{(a,a,0)}$, just as for the original complex $K_1(1, 2)$. The edges of $\K(0,2)$ are again face diagonals of the cubical tessellation with vertex-set $a\mathbb{Z}^3$; the faces are zigzags, again four around an edge. As for $\K(0,1)$, the four faces around an edge occur in two pairs of co-planar zigzags.  With $R_{0}$, $R_1$ and $\hat{R}_2$ as in \cite[eq. (6.1)]{pelsch}, we now find that the symmetry $R_{0}\hat{R}_2\cdot R_1$ of $F_2$ which ``shifts" the vertices of $F_2$ by one step along $F_2$, is a twist whose square is the translation by $(a,a,0)$. In particular, the base face of $\K(0,2)$ lies in the plane $x+y-z=0$ and is given by
\[ F_{2} = \{(0,a,a), (0,0,0),(a,0,a)\} + \mathbb{Z}\!\cdot\!(a,a,0) , \]
where here $(0,a,a)$ and $(a,0,a)$ are the two vertices of $F_2$ adjacent to the base vertex $(0,0,0)$. The vertex-figure of $\K(0,2)$ at $o$ coincides with the vertex-figure of $\K_1(2,1)$ at $o$, that is, with the edge-graph of a cuboctahedron.

Just like $\K(0,1)$, the complex $\K(0,2)$ is also closely related to the semiregular tessellation $\mathcal{S}$ of $\E$ by regular tetrahedra and octahedra described in Section~\ref{terba}. In fact, the zigzag base face of $\K(0,2)$ is a Petrie polygon of the regular tessellation of the plane $x+y-z=0$ by triangles formed from faces of $\mathcal{S}$, and each such Petrie polygon occurs as a face of $\K(0,2)$. Recall here that a {\em Petrie polygon\/} (or $1$-{\em zigzag}) is a path along edges such that any two, but no three, consecutive edges lie in a common face (see \cite[p. 196]{arp}). More generally, each Petrie polygon of a triangular tessellation induced by $\mathcal{S}$ on the affine hull of a triangular face of $\mathcal{S}$ is a face of $\K(0,1)$, and all faces of $\K(0,1)$ arise in this way.

\section{Complexes with mirror vector~$(0,k)$ and cyclic $G_2$}
\label{mirrcyc}

In this section we complete the enumeration of the simply flag-transitive regular complexes with mirror vector $(0,1)$ or $(0,2)$. In Section~\ref{mir0k} we exploited the operation $\lambda_0$ of (\ref{opone}) to deal with the case when $G_{2}$ is dihedral, and derived the corresponding complexes via $\lambda_0$ from suitable regular complexes with mirror vectors $(1,1)$ or $(1,2)$, respectively. Here we concern ourselves with the remaining case when $G_{2}$ is a cyclic group. In particular, we prove that this contributes no new regular complexes to our list. We already know from \cite[Section 7E]{arp} that polyhedra (with an irreducible symmetry group) cannot have mirror vector $(0,1)$ or $(0,2)$.

Let $\K$ be an infinite simply flag-transitive regular complex with mirror vector $(0,k)$, with $k=1$ or $2$, and let its symmetry group $G$ be irreducible. Then $R_0$ is the point reflection in the midpoint of the base edge $F_1$ of $\K$, and $R_{1}$ is a half-turn or plane reflection, depending on whether $k=1$ or $2$, with its mirror passing through the base vertex $F_0$ of $\K$. Now it is immediately clear that $\K$ must have (planar) zigzag faces. In fact, in the special group $G_*$ of $G$ we must have $R_0'=-I$ and hence $(R_{0}'R_{1})^{2}=I$, the identity mapping on $\E$, so modulo $G_*$ the basic cyclic symmetry $R_{0}R_{1}$ of the base face $F_2$ of $\K$ has only period $2$. Moreover, since $G_*$ contains the central inversion $-I$, we must have $G_{*}= [3,3]^{*}$ or $[3,4]$.

Now suppose the pointwise stabilizer $G_2$ of $F_1$ is a cyclic group generated by a rotation $S$ of period $r$. We show that all these data together already imply that $\K$ could only be the $2$-skeleton of a regular $4$-apeirotope in $\E$ (see (\ref{4apeirotopes})). However, this is impossible, since the latter is not a simply flag-transitive complex (see \cite[Section 4]{pelsch}). More precisely, we will establish that $G$ would have to be a flag-transitive subgroup of index $2$ in the full symmetry group of the $2$-skeleton of a regular $4$-apeirotope. Recall that the eight regular $4$-apeirotopes come in pairs of Petrie-duals, and that the apeirotopes in each pair have the same $2$-skeleton. Thus for our purposes it suffices to consider the apeirotopes $\apeir {\mathcal Q}$ with $\mathcal Q$ equal to $\{3,3\}$, $\{3,4\}$ or $\{4,3\}$.

\subsection{Mirror vector $(0,1)$ and cyclic $G_2$}

We begin with the case $k=1$. Suppose $\mathcal{K}$ is a simply flag-transitive regular complex with mirror vector $(0,1)$ and a cyclic group $G_2$ with rotational generator $S$. Recall our standing assumption that $G$ is irreducible. Then $R_0$ is a point reflection in the midpoint of the base edge $F_1$; $R_1$ is a half-turn about a line through the base vertex $F_{0}:=o$; and $S$ is a rotation about a line containing $F_1$. Moreover, as we remarked earlier, the faces of $\mathcal{K}$ are planar zigzags. In particular, this forces the rotation axis of $R_1$ to lie in the plane containing the base face $F_2$ (the axis of $R_1$ cannot be perpendicular to this plane, as no regular complex can have linear apeirogons as faces). Let $T_3$ denote the reflection in the plane through $F_2$. Then $T_3$ fixes each of $F_0$, $F_1$ and $F_2$ but does not lie in $G$; otherwise $T_3$ would actually belong to $G_2$ and make $G_2$ a dihedral group, contrary to our assumption (alternatively, as $T_3$ stabilizes the base flag, it would have to be trivial, by the simple flag-transitivity of $G$). However, we can also prove that $T_3$ must be a symmetry of $\mathcal{K}$, so $\mathcal{K}$ cannot have been simply-transitive. In fact, $\mathcal{K}$ must be the $2$-skeleton of a regular $4$-apeirotope in $\E$ since it has face-mirrors (for example, $T_3$), and $G$ must be a flag-transitive proper subgroup of the full symmetry group of $\mathcal{K}$, the latter being that of the $4$-apeirotope.

The proof hinges on the observation that $T_3$ commutes with the involutory generators $R_0$ and $R_1$ of $G$ and, up to taking inverses, with the (possibly non-involutory) generator $S$ as well, since $T_{3}ST_{3}=S^{-1}$. Note here that the mirror of $T_3$ contains the invariant point of $R_0$ and the rotation axes of $R_1$ and $S$. Now as the vertices, edges and faces of $\mathcal{K}$ are just the images of $F_{0}$, $F_{1}$ and $F_{2}$ under the elements of $G$, and these elements commute with $T_3$ up to taking inverses, we find that $T_3$ takes vertices, edges or faces of $\mathcal{K}$ to vertices, edges or faces of $\mathcal{K}$, respectively; more explicitly, if $R$ is any element in $G$ and $\widehat{R}$ its conjugate under $T_3$ (in the isometry group of $\E$), then $\widehat{R}$ lies in $G$ and $(F_{j}R)T_{3}=(F_{j}T_{3})\widehat{R}=F_{j}\widehat{R}$ for each $j$. Thus $T_3$ is actually a reflective symmetry of $\mathcal{K}$ leaving the plane through $F_2$ invariant; in particular, this plane is a face mirror of $\mathcal{K}$.

In summary, there are no simply flag-transitive regular complexes with mirror vector $(0,1)$ and a cyclic group $G_2$.

A more detailed analysis of the geometric situation above sheds some light on possible characterizations of $2$-skeletons of regular $4$-apeirotopes. In fact, there are just two possible choices for the special group $G_*$, namely $[3,3]^*$ and $[3,4]$, allowing for complexes $\K$ with $r=3$ or with $r=3$ or $4$, respectively. Suppose we pick the cube $C:=\{4,3\}$ with vertices $(\pm 1,\pm 1,\pm 1)$ as a reference figure for the action of $G_*$. Now if $G_{*}=[3,3]^*$ then $\K$ must necessarily be the $2$-skeleton of the regular $4$-apeirotope
\[ \apeir \{3, 3\} = \{\{\infty, 3\}_6 \# \{ \, \}, \{3, 3\}\} ,\]
with the tetrahedral vertex-figure $\mathcal{Q}:=\{3,3\}$ determined by one of the two sets of alternating vertices of $C$. On the other hand, if $G_{*}=[3,4]$ the outcome depends on the period of $S$ (that is, on $r$). If the period of $S$ is $3$, then $\K$ must be the $2$-skeleton of
\[ \apeir \{4, 3\} = \{\{\infty, 4\}_4 \# \{ \, \}, \{4, 3\}\}, \]
now with the cubical vertex-figure $\mathcal{Q}:=\{3,3\}$ given by $C$ itself. However, if the period of $S$ is $4$, then $\K$ must be the $2$-skeleton of
\[ \mathcal{P} = \apeir \{3,4\} = \{\{\infty, 3\}_6 \# \{ \, \}, \{3,4\}\} ,\]
with the octahedral vertex-figure $\mathcal{Q}:=\{3,4\}$ dually positioned to $C$ such that its vertices are at the centers of the faces of $C$. Notice that our choice of notation, $T_3$, for the reflection in the plane of $F_2$ was deliberate
to indicate its role as a generator of the symmetry group of the $4$-apeirotope.

\subsection{Mirror vector $(0,2)$ and cyclic $G_2$}
\label{mirr02cyc}

Similarly we deal with the mirror vector $(0,2)$. Let $\mathcal{K}$ be a simply flag-transitive regular complex with mirror vector $(0,2)$ and a cyclic group $G_2$ with rotational generator $S$. Then $R_0$ is the reflection in the midpoint of $F_1$; $R_1$ is a reflection in a plane through $F_{0}:=o$; and $S$ is a rotation about the line containing $F_1$. Now let $T_3$ denote the reflection in the plane through the (zigzag) base face $F_2$. Then this plane is perpendicular to the mirror of $R_1$ and also contains the invariant point of $R_0$ and the rotation axis of $S$. Thus  $T_3$ stabilizes the base flag $\{F_{0},F_{1},F_{2}\}$. However, $T_3$ cannot belong to $G$, since otherwise $G_2$ would be dihedral, not cyclic. On the other hand, $T_3$ again commutes with the generators $R_0$ and $R_1$ of $G$ and, up to taking inverses, with the generator $S$ as well (that is $T_{3}S=S^{-1}T_{3}$); note here that the mirror of $T_3$ is perpendicular to the mirror of $R_1$ and contains the mirrors of $R_0$ and $S$. It follows that we can proceed exactly as for the mirror vector $(0,1)$ to show that $T_3$ is actually a reflective symmetry of $\mathcal{K}$ leaving the plane through $F_2$ invariant and making it a face mirror. Hence, $\mathcal{K}$ cannot have been simply-transitive and must be the $2$-skeleton of a regular $4$-apeirotope in $\E$. Moreover, $G$ must be a flag-transitive proper subgroup of the full symmetry group of $\mathcal{K}$, the latter being that of the underlying $4$-apeirotope.

Thus no simply flag-transitive regular complex can have a mirror vector $(0,2)$ and a cyclic group $G_2$.

A further analysis shows that the operation
\begin{equation}\label{opt3}
\lambda \!:\;\,  (R_0, R_1, S)\; \mapsto\; (R_0, T_3R_1, S)
\end{equation}
on the generators of the underlying groups interchanges the two possible choices of mirror vectors $(0,1)$ and $(0,2)$ if $G_2$ is cyclic. Thus, in some sense, these cases are equivalent. The operation (\ref{opt3}) actually applies to the symmetry group of the corresponding regular $4$-apeirotope, where it corresponds to performing the Petrie operation (on the vertex-figure). As the eight regular $4$-apeirotopes in $\E$ come in pairs of Petrie duals sharing a common $2$-skeleton (see \cite[Theorem~4.3]{pelsch}), the three apeirotopes with~zigzag faces associated with the complexes with mirror vector $(0,2)$ (and cyclic $G_2$) occur here in the form $\apeir {\mathcal Q}$ with $\mathcal{Q}=\{4,3\}_3$, $\{6,4\}_3$ or $\{6,3\}_4$, these being the Petrie duals of $\{3,3\}$, $\{3,4\}$ or $\{4,3\}$, respectively.
\bigskip

The following theorem summarizes our discussion for the mirror vectors $(0,k)$ with $k=1,2$. Note that there are also regular polyhedra with these mirror vectors; for example, the Petrie duals of the regular plane tessellations $\{4,4\}$, $\{3,6\}$ and $\{6,3\}$ have mirror vector $(0,1)$ when viewed in the plane.

\begin{theorem}
\label{classif0k}
Apart from polyhedra, the complexes $\K(0,1)$ and $\K(0,2)$ described in Sections~\ref{k01} and~\ref{k02} are the only simply flag-transitive regular polygonal complexes with mirror vectors $(0,1)$ or $(0,2)$, respectively.
\end{theorem}

\section{Complexes with mirror vector~$(2,k)$}
\label{mir2k}

Our last step is the enumeration of the simply flag-transitive complexes $\K$ with mirror vector $(2,k)$. According to Lemma~\ref{appone}, each such complex can be obtained from a regular complex $\mathcal{L}$ which either has face mirrors or is simply flag-transitive with mirror vector $(0,k)$, by applying to $\mathcal{L}$ the operation $\lambda_0$ with respect to a half-turn in $G_2(\mathcal{L})$.

The next lemma, applied with $\mathcal{L}=\K^{\lambda_0}$ and $\mathcal{L}'=\K$, rules out the possibility that the complex $\mathcal{L}$ has face-mirrors.

\begin{lemma}\label{facemirlambda02}
Let $\mathcal{L}$ be a regular complex with face mirrors. Assume that $G_0(\mathcal{L})$ contains a point reflection $R_0$. Let $\mathcal{L}'$ denote the regular complex obtained from $\mathcal{L}$ by the operation $\lambda_0(R_2)$, where $R_2$ is the unique half-turn in $G_2(\mathcal{L})$. Then $\mathcal{L}'$ also has face mirrors.
\end{lemma}

\begin{proof}
The proof is similar to that of Lemma \ref{facemirlambda01}. It suffices to note that all elements in $G_1(\mathcal{L})$ fix the plane $L$ through the base face of $\mathcal{L}$, and that $R_0 R_2$ is a reflection in a plane perpendicular to $L$. Then it follows that the base face of $\mathcal{L}'$ also lies in $L$, so $\mathcal{L}'$ also has face-mirrors.
\end{proof}

Thus we may concentrate on the simply flag-transitive complexes $\mathcal{L}$ with mirror vector $(0,k)$. We know from Theorem~\ref{classif0k} that there is just one such complex for each $k$.

\subsection{Complexes with mirror vector~$(2,1)$}

We first derive the unique simply flag-transitive complex with mirror vector $(2,1)$ from $\mathcal{L}=\K(0,1)$ by means of the operation $\lambda_0(R_2)$ associated with the half-turn $R_2$ in the dihedral group $G_2(\mathcal{L})$:
\begin{equation}
\K(2, 1) := \K(0, 1)^{\,\lambda_0 (R_2)}.
\end{equation}

The vertex-set of $\K(2,1)$ is again $\Lambda_{(a,a,0)}$, as for the original complex $\K(0,1)$. The edges of $\K(0,1)$ are face diagonals of the standard cubical tessellation with vertex-set $a\mathbb{Z}^3$, and the faces are convex regular hexagons in planes perpendicular to main diagonals of $a\mathbb{Z}^3$. There are four faces around an edge, such that opposite faces are co-planar. Computing the distinguished generators for $\K(2,1)$ from those of $\K(0,1)$ (which, in turn, are based on the generators for $\K_5(1,1)$ described in (\ref{genk511})), we find that the vertices of the hexagonal base face of $\K(2,1)$ are given by
\[(0,0,0), (a,0,a), (a,a,2a), (0,2a,2a), (-a,2a,a), (-a,a,0),\]
in this order. The base face is centered at the point $(0,a,a)$ of $\Lambda_{(a,a,0)}$ and forms an equatorial hexagon of the cuboctahedron whose vertices are the midpoints of the edges of the $2a\!\times\!2a\!\times\!2a$ cube with center at $(0,a,a)$. Note that $(0,a,a)$ is the common center (but not a vertex) of four faces of $\K(2,1)$, each an equatorial hexagon of the  cuboctahedron just described. The vertex-figure of $\K(2,1)$ at $o$ coincides with the vertex-figure of $\K(0,1)$ at $o$, that is, with the edge graph of a non-standard cuboctahedron with skew square faces.

As with the complexes $\K(0,1)$ and $\K(0,2)$, the geometry of $\K(2,1)$ can be described in terms of the semiregular tessellation $\mathcal{S}$ of $\E$ by regular tetrahedra and octahedra (see Section~\ref{terba}). The hexagonal base face of $\K(2,1)$ also lies in the plane $x+y-z=0$ and is just a $2$-hole of the regular tessellation of this plane by triangles formed from faces of $\mathcal{S}$; in fact, each $2$-hole of this tessellation is a face of $\K(2,1)$. Recall here that a $2$-{\em hole\/}, or simply a {\em hole\/}, of any regular map on a surface, is a path along edges that successively take the second exit on the left (in a local orientation), at each vertex. For a regular tessellation by triangles, the $2$-holes are just the boundary edge-paths of convex hexagons comprising the six triangles with a common vertex. More generally, each face of $\K(2,1)$ is a hole of the regular tessellation of its affine hull by triangles formed from faces of $\mathcal{S}$, and each such hole is a face of~$\K(2,1)$.

\subsection{Complexes with mirror vector~$(2,2)$}
\label{mirr22com}

Finally, then, we construct the unique simply flag-transitive complex with mirror vector $(2,2)$ from $\mathcal{L}=\K(0,2)$ by applying the operation $\lambda_0 (R_2)$, where again $R_2$ denotes the unique half-turn in the dihedral group $G_2(\mathcal{L})$:
\begin{equation}
\K(2, 2) \,:=\, \K(0, 2)^{\,\lambda_0 (R_2)}.
\end{equation}

As for the original complex $\K(0, 2)$, the vertex-set of $\K(2,2)$ is $\Lambda_{(a,a,0)}$ and the edges are again face diagonals of the cubical tessellation with vertex-set $a\mathbb{Z}^3$. The faces are triangles, four around each edge such that opposite triangles are co-planar. Now the distinguished generators for $\K(2,2)$ can be obtained from those of $\K(0,2)$, which, in turn, are based on the generators for $\K_1(1,2)$ described in \cite[eq. (6.1)]{pelsch}. The base face of $\K(2,2)$ lies in the plane $x+y-z=0$ and has vertices
\[(0,0,0), (a,0,a), (0,a,a).\]
The vertex-figure of $\K(2,1)$ at $o$ coincides with the vertex-figure of $\K(0,2)$ at $o$, that is, with the edge-graph of a cuboctahedron.

The complex $\K(2,2)$ can best be visualized as the $2$-skeleton of the semiregular tessellation $\mathcal{S}$ of $\E$ by regular tetrahedra and octahedra described in Section~\ref{terba}. The faces of $\mathcal{S}$ are regular triangles, each shared by an octahedral tile and a tetrahedral tile of $\mathcal{S}$. It is straightforward to check that $\K(2,2)$ is just the $2$-skeleton of $\mathcal{S}$.
\bigskip

Inspection of the list of simply flag-transitive regular complexes shows that there are just two complexes with finite planar (in fact, convex) faces, namely $\K(2,2)$ with triangular faces and $\K(2,1)$ with hexagonal faces (see also Table~\ref{tabone}). The complex $\K(2,2)$ can be viewed as the $2$-skeleton of the semiregular tessellation $\mathcal{S}$ by regular tetrahedra and octahedra, and hence is a geometric complex embedded (without self-intersections) in $\E$. On the other hand, $\K(2,1)$ can not be viewed as a geometric complex embedded in $\E$. In fact, every vertex of $\K(2,1)$ is the common center of the four faces of $\K(2,1)$ given by the equatorial hexagons of a suitable cuboctahedron with this vertex as center; when the hexagonal faces of $\K(2,1)$ are viewed as bounding convex hexagons, these four hexagons all intersect in their common center, and hence self-intersections do occur in this case.
\bigskip

In conclusion, we have established the following theorem.

\begin{theorem}
\label{classif2k}
Apart from polyhedra, the complexes $\K(2,1)$ and $\K(2,2)$ described in this section are the only simply flag-transitive regular polygonal complexes with mirror vectors $(2,1)$ or $(2,2)$, respectively.
\end{theorem}

\section{The enumeration} 

The following theorem summarizes our enumeration of regular polygonal complexes in euclidean $3$-space.

\begin{theorem}
\label{fullclassif}
Up to similarity, there are exactly 25 regular polygonal complexes in $\E$ which are not regular polyhedra, namely 21 simply flag-transitive complexes and $4$ complexes which are $2$-skeletons of regular $4$-apeirotopes in~$\E$.
\end{theorem}

Table~\ref{tabone} organizes the $21$ simply flag-transitive polygonal complexes by mirror vectors and includes for each complex the data about the pointwise edge stabilizer $G_2$, the number $r$ of faces surrounding an edge, the structure of faces and vertex-figures, the vertex-set, and the structure of the special group. The symbols $p_c$, $p_s$, $\infty_2$, or $\infty_k$ with $k=3$ or $4$, respectively, in the face column indicate that the faces are {\em convex} $p$-gons, {\em skew} $p$-gons, planar zigzags, or helices over {\em $k$-gons}. (In some sense, a planar zigzag is a helix over a $2$-gon, hence our notation. Clearly, the suffix in $3_c$ is redundant.)  We also set
\[ V_{a}:=a\mathbb{Z}^{3}\!\setminus\! ((0,0,a)\!+\!\Lambda_{(a,a,a)}),\;\;
W_{a}:= 2\Lambda_{(a,a,0)} \cup ((a,-a,a)\!+\!2\Lambda_{(a,a,0)}), \]
to have a short symbol available for some of the vertex-sets. The vertex-figures of polygonal complexes are finite graphs, so an entry in the vertex-figure column describing a solid figure is meant to represent the edge-graph of this figure, with ``double" indicating the double edge-graph. The abbreviation ``ns-cuboctahedron" stands for (the edge graph of) the ``non-standard cuboctahedron" (as explained earlier in the text).

\begin{table}[htb]
\centering
{\begin{tabular}{|c|c|c|c|c|c|c|c|}  \hline
mirror  & complex&$G_2$ & $r$ &face  &vertex- &vertex- & special\\
vector &               &            &       &          & figure&set     & group  \\[.05in]
\hline
\hline
$(1,2)$  & $\K_1(1,2)$& $D_2$  & $4$ &$4_s$ & cuboctahedron&$\Lambda_{(a,a,0)}$&$[3,4]$\\
\hline
             & $\K_2(1,2)$& $C_3$& $3$ &$4_s$ & cube&$\Lambda_{(a,a,a)}$&$[3,4]$\\
\hline
             & $\K_3(1,2)$& $D_3$& $6$   &$4_s$ &double cube&$\Lambda_{(a,a,a)}$&$[3,4]$\\
\hline
             & $\K_4(1,2)$& $D_2$& $4$ & $6_s$&octahedron&$a\mathbb{Z}^3$&$[3,4]$\\
\hline
             & $\K_5(1,2)$& $D_2$& $4$ &$6_s$&double square&$V_a$&$[3,4]$\\
\hline
             & $\K_6(1,2)$& $D_4$& $8$ &$6_s$&double octahedron&$a\mathbb{Z}^3$&$[3,4]$\\
\hline
             & $\K_7(1,2)$& $D_3$& $6$ &$6_s$&double tetrahedron&$W_a$&$[3,4]$\\
\hline
             & $\K_8(1,2)$& $D_2$& $4$ &$6_s$&cuboctahedron&$\Lambda_{(a,a,0)}$&$[3,4]$\\
\hline
\hline
$(1,1)$  & $\K_1(1,1)$& $D_3$& $6$ &$\infty_3$&double cube&$\Lambda_{(a,a,a)}$&$[3,4]$\\
\hline
             & $\K_2(1,1)$& $D_2$& $4$ &$\infty_3$&double square&$V_a$&$[3,4]$\\
\hline
             & $\K_3(1,1)$& $D_4$& $8$ &$\infty_3$&double octahedron&$a\mathbb{Z}^3$ & $[3,4]$ \\
\hline
             & $\K_4(1,1)$& $D_3$ & $6$ &$\infty_4$& double tetrahedron&$W_a$ &$[3,4]$\\
\hline
             & $\K_5(1,1)$& $D_2$& $4$ &$\infty_4$&ns-cuboctahedron&$\Lambda_{(a,a,0)}$&$[3,4]$ \\
\hline
             & $\K_6(1,1)$& $C_3$& $3$ &$\infty_4$&tetrahedron& $W_a$&$[3,4]^+$\\
\hline
             & $\K_7(1,1)$& $C_4$& $4$ &$\infty_3$&octahedron&$a\mathbb{Z}^3$ & $[3,4]^+$ \ \\
\hline
             & $\K_8(1,1)$& $D_2$& $4$ &$\infty_3$&ns-cuboctahedron&$\Lambda_{(a,a,0)}$&$[3,4]$ \\
\hline
             & $\K_9(1,1)$& $C_3$& $3$ &$\infty_3$&cube&$\Lambda_{(a,a,a)}$&$[3,4]^+$ \\
\hline
\hline
$(0,1)$  & $\K(0,1)$& $D_2$& $4$ &$\infty_2$&ns-cuboctahedron&$\Lambda_{(a,a,0)}$&$[3,4]$ \\
\hline
\hline
$(0,2)$  & $\K(0,2)$& $D_2$& $4$ &$\infty_2$&cuboctahedron& $\Lambda_{(a,a,0)}$&$[3,4]$ \\
\hline
\hline
$(2,1)$  & $\K(2,1)$& $D_2$& $4$ & $6_c$ &ns-cuboctahedron&$\Lambda_{(a,a,0)}$&$[3,4]$ \\
\hline
\hline
$(2,2)$  & $\K(2,2)$& $D_2$& $4$ &$3_c$&cuboctahedron&$\Lambda_{(a,a,0)}$&$[3,4]$\\
\hline
\hline
\end{tabular}
\medskip
\caption{The 21 simply flag-transitive regular polygonal complexes in $\E$ which are not regular polyhedra.}
\label{tabone}}
\end{table}

\section{Subcomplex relationships} 

In this last section we provide the full net of subcomplex relationships for regular polygonal complexes in space, including regular polyhedra. In each diagram, a vertical or slanted line indicates that the complex $\K$ at the bottom is a subcomplex  of the complex $\mathcal{L}$ at the top, or equivalently, that the complex $\mathcal{L}$ at the top is a ``compound" of congruent copies of the complex $\K$ at the bottom; the label attached to the line is the number of congruent copies of $\K$ in this representation of $\mathcal{L}$ as a compound, allowing~$\infty$. 

Across the bottom of a diagram we usually find polyhedra and occasionally (indecomposable) complexes from among those discussed in this paper or in \cite{pelsch}. Polyhedra can only occur at the bottom of a diagram. Two diagrams have three layers and each contains the diagrams for the complexes in the middle layer as subdiagrams (these diagrams are not listed separately). 

The regular polyhedra in $\E$ were enumerated in Gr\"unbaum~\cite{gr} and Dress~\cite{d1,d2}. We refer the reader to McMullen \& Schulte~\cite[Section 7E]{arp} (or \cite{ordinary}) for a quick method of arriving at the full characterization, and for the notation for polyhedra used in the diagrams. For a regular $4$-apeirotope $\mathcal{P}$ we write $skel_{2}(\mathcal{P})$ for the $2$-skeleton of $\mathcal{P}$ (see Section~\ref{terba}). 

The complexes are grouped according to their type of faces, beginning with the complexes with finite, planar or skew, faces and followed by the complexes with infinite, zigzag or helical, faces.
\bigskip

\noindent
PLANAR FACES\newline
\vskip.08in

\input{diagram1.pic}
\vspace{.1in}

\input{diagram2.pic}

\vspace{0.7cm}

\noindent
SKEW FACES\newline
\vskip.08in

\input{diagram6.pic}
\vspace{.1in}

\input{diagram7.pic}
\vspace{.1in}

\input{diagram8.pic}

\vspace{0.7cm}

\noindent
ZIGZAG FACES\newline
\vskip.08in

\input{diagram3.pic}
\vspace{.1in}

\input{diagram4.pic}
\vspace{.1in}

\input{diagram5.pic}

\vspace{0.7cm}

\noindent
HELICAL FACES\newline
\vskip.08in

\input{diagram9.pic}
\vspace{.1in}

\input{diagram10.pic}
\vspace{.1in}

\input{diagram11.pic}

\vspace{0.2cm}

The preceding diagrams provide the full net of {\em geometric\/} subcomplex relationships for regular polygonal complexes in $\E$. We have not investigated the question if there are any other {\em combinatorial\/} subcomplex relationships. This of course is closely related to the problem of determining the full combinatorial automorphism group for each complex, which in general could be larger than the symmetry group and in particular have a larger flag stabilizer. For example, a (geometrically) simply flag-transitive polygonal complex may not be combinatorially simply flag-transitive. Clearly, this does not occur for regular polyhedra, where the two groups are isomorphic. 

It would also be interesting to know if the geometrically distinct regular polygonal complexes described here are also combinatorially distinct, or if two of these complexes can be combinatorially isomorphic in a non-geometric way. We conjecture that the regular polygonal complexes in $\E$ are indeed also combinatorially distinct. (We are ignoring here the case of blended polyhedra, where the relative size of the components of the blend provides a continuous parameter for the geometric realizations.)

\vskip.1in
\noindent
{\bf Acknowledgment}
We are grateful to an anonymous referee for a thoughtful review with valuable comments. 
\bibliographystyle{amsplain}

\begin{thebibliography}{999}

\bibitem{ar} J.L.Arocha, J.Bracho and L.Montejano, {\em Regular projective polyhedra with planar faces, Part I\/}, Aequat.\ Math.\  59 (2000), 55--73.

\bibitem{bra} J.Bracho, {\em Regular projective polyhedra with planar faces, Part II\/}, Aequat.\ Math.\ 59 (2000), 160--176.

\bibitem{crsp} H.S.M.Coxeter, {\em Regular skew polyhedra in 3 and 4 dimensions and their topological analogues\/},  Proc.\ London Math.\ Soc. (2) 43 (1937), 33--62.  (Reprinted with amendments in {\em Twelve Geometric Essays\/}, Southern Illinois University Press (Carbondale, 1968), 76--105.)

\bibitem{coxsr} H.S.M.Coxeter, {\em Regular and semi-regular polytopes, I\/}, Mathematische Zeitschrift 4 (1940), 380--407.

\bibitem{coxeter} H.S.M.Coxeter, {\em Regular Polytopes\/} (3rd edition),  Dover
(New York, 1973).

\bibitem{del} O.Delgado-Friedrichs, M.D.Foster, M.O'Keefe, D.M.Proserpio, M.M.J.Treacy and O.M.Yaghi, {\em What do we know about three-periodic nets?\/}, J. Solid State Chemistry 178 (2005), 2533--2554.

\bibitem{d1} A.W.M.Dress, {\em A combinatorial theory of Gr\"unbaum's new regular polyhedra, I:  Gr\"unbaum's new regular polyhedra and their automorphism group\/},  Aequationes Math.\  23 (1981), 252--265.

\bibitem{d2} A.W.M.Dress, {\em A combinatorial theory of Gr\"unbaum's new regular polyhedra, II:  complete enumeration\/},  Aequationes Math.\ 29 (1985), 222--243.

\bibitem{gr1} B.Gr\"unbaum, {\em Regular polyhedra --- old and new\/}, Aequat.\ Math.\ {16} (1977), 1--20.

\bibitem{grhol} B.Gr\"unbaum, {\em Polyhedra with hollow faces\/},  in {\em
Polytopes:  Abstract, Convex and Computational\/} (eds.\ T.~Bisztriczky,
P.~McMullen, R.~Schneider and A.~Ivi\'c Weiss), NATO ASI Series C 440,
Kluwer (Dordrecht etc., 1994), 43--70.

\bibitem{gr} B.Gr\"unbaum, {\em Acoptic polyhedra\/}, In {\em Advances in Discrete and
Computational Geometry\/}, B.Chazelle et al. (ed.), Contemp.\ Math.\ 223,
American Mathematical Society (Providence, RI, 1999), 163--199.

\bibitem{gruni} B.Gr\"unbaum, {\em Uniform tilings of 3-space\/}, Geombinatorics 4 (1994), 49--56.

\bibitem{jo} N.W.Johnson, {\em Uniform Polytopes\/},  Cambridge University Press
(Cambridge, to appear).

\bibitem{pm} P.McMullen, {\em Regular polytopes of full rank\/}, Discrete \& Computational Geometry 32 (2004), 1--35.

\bibitem{pm1} P.McMullen, {\em Four-dimensional regular polyhedra\/}, Discrete \& Computational Geometry 38 (2007), 355--387.

\bibitem{pm2} P.McMullen, {\em Regular apeirotopes of dimension and rank 4\/}, Discrete \& Computational Geometry 42 (2009), 224--260.

\bibitem{ordinary} P.McMullen and E.Schulte, {\em Regular polytopes in ordinary space\/}, Discrete Comput. Geom. 17 (1997), 449--478.

\bibitem{arp}
P.McMullen and E. Schulte, \emph{Abstract regular polytopes}, Encyclopedia of Mathematics and its Applications, Vol. 92, Cambridge University Press, Cambridge, UK, 2002.

\bibitem{ms3} P.McMullen and E.Schulte, {\em Regular and chiral polytopes in low dimensions\/}, In {\em The Coxeter Legacy -- Reflections and Projections\/} (eds. C.Davis and E.W.Ellers), Fields Institute Communications, Volume 48, American Mathematical Society (Providence, RI, 2006), 87--106.

\bibitem{monw5} B.R.Monson and A.I.Weiss, {\em Realizations of regular toroidal maps\/},
Canad.\ J.\ Math. (6) 51 (1999), 1240--1257.

\bibitem{okee}
M.O'Keeffe, {\em Three-periodic nets and tilings:\  regular and related infinite polyhedra\/}, Acta Crystallographica  A 64 (2008), 425--429.

\bibitem{okhy} M.O'Keeffe and B.G.Hyde, {\em Crystal Structures; I. Patterns and Symmetry\/}, Mineralogical Society of America, Monograph Series, Washington, DC, 1996.

\bibitem{pelsch} D.Pellicer and E.Schulte, {\em Regular polygonal complexes in space, I}, Trans. Amer. Math. Soc. 362 (2010), 6679--6714.

\bibitem{pelwei} D.Pellicer and A.I.Weiss, {\em Combinatorial structure of Schulte's chiral polyhedra\/}, Discrete \& Comput. Geom. 44 (2010), 167--194.

\bibitem{chiral1}
E.Schulte, {\em Chiral polyhedra in ordinary space, I\/}, Discrete Comput. Geom. 32
(2004), 55--99.

\bibitem{chiral2}
E.Schulte, {\em Chiral polyhedra in ordinary space, II\/}, Discrete Comput. Geom. 34
(2005), 181--229.

\bibitem{wells} A.F.Wells, {\em Three-dimensional Nets and Polyhedra\/},
Wiley-Interscience (New York, etc., 1977).

\end{thebibliography}

\end{document}